\documentclass[a4paper]{amsart}

\usepackage{srcltx} 
\usepackage{amsmath,amsthm}
\usepackage{mathrsfs}
\usepackage{amssymb}
\usepackage{enumitem}
\usepackage[usenames,dvipsnames]{pstricks}
\usepackage{epsfig}
\usepackage{esint}    
\usepackage{pst-grad} 
\usepackage{pst-plot} 
\usepackage{pinlabel}

\usepackage{hyperref}

\newtheorem{theor}{Theorem}[section]

\newtheorem{lemma}[theor]{Lemma}
\newtheorem{prop}[theor]{Proposition}

\newtheorem*{question*}{Question}

\theoremstyle{definition}
\newtheorem{defn}{Definition}

\theoremstyle{remark}

\newtheorem{remark}[theor]{Remark}

\numberwithin{equation}{section}
\numberwithin{defn}{section}

\newcommand{\R}{\mathbb{R}}        
\renewcommand{\S}{\mathbb{S}}        

\newcommand{\RNp}{\mathbb{R}^{m+1}_{+}}
\newcommand{\X}{\mathbb{X}^{2s}(\RNp)}
\newcommand{\eps}{\varepsilon}

\renewcommand{\d}{\mathrm{d}}
\renewcommand{\a}{1-2s}
\newcommand{\abs}[1]{\left| #1 \right|}
\newcommand{\norm}[1]{\left\| #1 \right\|}

\newcommand{\extp}{\mathchoice{{\textstyle\bigwedge}}%
    {{\bigwedge}}%
    {{\textstyle\wedge}}%
    {{\scriptstyle\wedge}}}

\usepackage{color}
\usepackage[normalem]{ulem} 

\newcommand{\G}{\mathcal{G}} 
\newcommand{\E}{\mathcal{E}}
\newcommand{\ve}{\varepsilon}
\newcommand{\h}{{\mathbf{h}}}

\title[]{Partial regularity of the heat flow of half-harmonic maps and applications to harmonic maps with free boundary}
\author{Ali Hyder}
\author{Antonio Segatti}
\author{Yannick Sire}
\author{Changyou Wang}
\date{}

\begin{document}

\begin{abstract}
We introduce a heat flow associated to half-harmonic maps, which have been introduced by Da Lio and Rivi\`ere. Those maps exhibit integrability by compensation in one space dimension and are related to harmonic maps with free boundary. We consider a new flow associated to these harmonic maps with free boundary which is actually motivated by a rather unusual heat flow for half-harmonic maps. We construct then weak solutions and prove their partial regularity in space and time via a Ginzburg-Landau approximation. The present paper complements the study initiated by Struwe and Chen-Lin. 
\end{abstract}

\maketitle

\tableofcontents

\section{Introduction}

In \cite{DR,DR2}, Da Lio and Rivi\`ere introduced a new notion of harmonic map by considering critical points of a Gagliardo-type $\dot H^s(\mathbb R^d)$ semi-norm in the conformal case $s=\frac12$ and $d=1$. Those maps have found a geometric application in the works of Fraser and Schoen about extremal metrics of Steklov eigenvalues (see e.g. \cite{fraserSchoen} and references therein). These maps correspond to an {\sl extrinsinc} version of harmonic maps with free boundary as proved by Millot and Sire in \cite{MS}. On the other hand, Moser \cite{moser} introduced an {\sl intrinsic } version of those latter maps, and Roberts \cite{roberts} investigated regularity of generalized versions of those maps, i.e. considering Gagliardo functionals for any $s \in (0,1)$. Whenever the extrinsic version of those maps is concerned, critical points of the functional introduced by Da Lio and Rivi\`ere satisfy the following equation in the distributional sense 
$$
(-\Delta )^\frac12 u \perp T_{u} N
$$
whenever $u: \mathbb S^1 \to N$. As pointed out in \cite{MS}, the harmonic extension of those maps into the unit disk are so-called harmonic maps with free boundary. We now introduce such maps in a general setup: let $(M, g)$ be an $m$-dimensional smooth Riemannian manifold with boundary $\partial M$ and $N$ be another  smooth compact Riemannian manifold without boundary. Suppose $\Sigma$ is a $k$-dimensional submanifold of $N$ without boundary. Any continuous map $u_0: M\to N$ satisfying $u_0(\partial M)\subset \Sigma$ defines a relative homotopy class in maps from $(M, \partial M)$ to $(N, \Sigma)$. A map $u: M\to N$ with $u(\partial M)\subset \Sigma$ is called homotopic to $u_0$ if there exists a continuous homotopy $h:[0, 1]\times M \to N$ satisfying $h([0,1]\times \partial M)\subset \Sigma$, $h(0) = u_0$ and $h(1) = u$. An interesting problem is that whether or not each relative homotopy class of maps has a representation by harmonic maps, which is equivalent to the following problem:
\begin{equation}\label{e:harmonicmapwithfreeboundary}
\begin{cases}
-\Delta u = \Gamma(u)(\nabla u, \nabla u),\\
u(\partial M)\subset \Sigma,\\
\frac{\partial u}{\partial \nu}\perp T_u\Sigma.
  \end{cases}
\end{equation}
Here $\nu$ is the unit normal vector of $M$ along the boundary $\partial M$, $\Delta \equiv \Delta_M$ is the Laplace-Beltrami operator of $(M,g)$, $\Gamma$ is the second fundamental form of $N$ (viewed as a submanifold in $\mathbb{R}^\ell$ via Nash's isometric embedding), $T_pN$ is the tangent space in $\mathbb{R}^\ell$ of $N$ at $p$ and $\perp$ means orthogonal in $\mathbb{R}^\ell$. (\ref{e:harmonicmapwithfreeboundary}) is the Euler-Lagrange equation for critical points of the Dirichlet energy functional
\begin{equation*}
E(u) = \int_{M}|\nabla u|^2\,dv_g
\end{equation*}
defined over the space of maps
\begin{equation*}
H^1_\Sigma(M, N) = \{u\in H^1(M, N): u(x)\subset \Sigma \ {\rm{a.e.}}\ x\in\partial M \}.
\end{equation*}
Here $H^1(M, N)=\big\{u\in H^1(M,\mathbb R^\ell): \ u(x)\in N \ {\rm{a.e.}}\ x\in M\big\}$. Both the existence and partial regularity of energy minimizing harmonic maps in $H^1_\Sigma(M, N)$ have been established (for example, in \cite{BaldesMM1982}, \cite{DuzaarSteffen1989JRAM}, \cite{DuzaarSteffenAA1989}, \cite{GulliverJostJRAM1987}, \cite{HardtLinCPAM1989}). A classical approach to investigate (\ref{e:harmonicmapwithfreeboundary}) is to study the following parabolic problem
\begin{equation}\label{e:harmonicmapflowwithfreeboundary}
\begin{cases}
\partial_t u -\Delta u = \Gamma(u)(\nabla u, \nabla u)&\text{ on }M\times [0, \infty),\\
u(x, t)\in \Sigma &\text{ on } \partial M\times [0,\infty),\\
\frac{\partial u}{\partial \nu}(x, t)\perp T_{u(x, t)}\Sigma &\text{ on}\ \partial M\times [0,\infty)\\
u(\cdot, 0) = u_0 &\text{ on }M.
  \end{cases}
\end{equation}
This is the so-called harmonic map flow with free boundary. (\ref{e:harmonicmapflowwithfreeboundary}) was first studied by Ma \cite{MaLiCMH1991} in the case $m = dim M = 2$, where a global existence and uniqueness result for finite energy weak solutions was obtained under {suitable} geometrical hypotheses on $N$ and $\Sigma$. Global existence for weak solutions of (\ref{e:harmonicmapflowwithfreeboundary}) was established by Struwe in \cite{StruweManMath1991} for $m\ge 3$. 
In \cite{Hamilton1975LNM}, Hamilton considered the case when $\partial N = \Sigma$ is totally geodesic and the sectional curvature $K_N\leq 0$.
He proved the existence of a unique global smooth solution for (\ref{e:harmonicmapflowwithfreeboundary}). When $N$ is an Euclidean space, the first equation in (\ref{e:harmonicmapflowwithfreeboundary}) is the standard heat equation
\begin{equation}
u_t - \Delta u = 0 \text{ on }M \times [0, \infty).
\end{equation}
As pointed out in \cite{chen-lin} and \cite{StruweManMath1991}, estimates near the boundary for (\ref{e:harmonicmapflowwithfreeboundary}) are difficult because of the highly nonlinear boundary conditions. Struwe in \cite{StruweManMath1991} introduced the heat flow for the intrinsic version of harmonic maps with free boundary. In particular, he used a Ginzburg-Landau approximation in the interior, hence keeping the boundary condition highly nonlinear.

 In the present paper we revisit the Struwe approximation argument by considering a natural, though unusual, heat flow associated to the equation derived by Da Lio and Rivi\`ere, that we called {\sl  half-harmonic  maps}.  Wettstein \cite{wettstein1,wettstein2} considered the natural $L^2-$gradient flow of the $\dot H^\frac12$-energy of half-harmonic map defined distributionally by 
 
\begin{equation}\label{eq:fhf1}
\partial_t u +(-\Delta)^\frac12 u \perp T_{u} N \ {\rm{in}}\ \mathbb R \times [0,\infty),
\end{equation}
where $\partial_t  +(-\Delta)^\frac12$ is the so-called Poisson operator whose expression is explicit.  Some weak solutions for this flow have been constructed in \cite{SSW}. Infinite-time blow up has been considered in \cite{SWZ}.

As far as the (partial) regularity of the heat flow of harmonic maps is concerned, a way to construct weak solutions  is to have a suitable monotonicity formula for a Ginzburg-Landau approximation of the system (see the monograph \cite{bookLW} for an up to date  account). At the moment such a  monotonicity formula is {\sl not } available for the latter system \eqref{eq:fhf1}, despite this flow being the natural one analytically. 

Therefore, we replace the previous flow by 
\begin{equation}
\label{eq:fhf}
\begin{cases}
\big(\partial_t  - \Delta \big)^\frac12 u \perp T_u N &\hbox{ in } \R^m \times (0,+\infty),\\
u(x,t) = u_0(x,t)&\hbox{ in } \R^m\times (-\infty, 0].
\end{cases}
\end{equation}

Clearly, these two flows admit the same stationary solutions, which are (weak) half-harmonic maps into $N$. However, it is known (see \cite{banerjeeGaro}) that, suitably formulated, the flow \eqref{eq:fhf} does enjoy a monotonicity formula.  This is due to the existence of a suitable (caloric) extension to the upper-half space (see \cite{NS} and \cite{ST}). As we will see below, though the operator $\big(\partial_t  - \Delta \big)^\frac12$ defined as a Fourier-Laplace multiplier seems unnatural, its caloric extension to the upper half-space is naturally associated to extrinsic harmonic maps with free boundary. Considering a Ginzburg-Landau approximation {\sl at the boundary}, which is more in the spirit of the approach by Da Lio and Rivi\`ere and motivated by the Ginzburg-Landau approximation of extrinsic harmonic maps with free boundary proved in \cite{MS},  we construct weak solutions which are partially regular. 

We will always assume in the following that $(M,g)=(\mathbb R^m, dx^2)$. To keep the technicalities as simple as possible we will present the detailed proof for the case that the target manifold is a sphere, and provide necessary modifications of proof for general target manifolds
$N$ in Appendix B. Let $(\mathbb S^{\ell-1},g_{can})$ be the $(\ell-1)$ dimensional unit sphere in $\R^\ell$ equipped with the standard metric. Given $u_0:\mathbb{R}^m\times(-\infty, 0]\to \mathbb S^{\ell-1}$ with $u_0(\cdot, t)\in \dot{H}^s(\R^m)$ for $t\le 0$,
we introduce the following evolution: for $(X,t)=(x,y,t)\in\mathbb{R}^{m+1}_{+}\times \mathbb{R}$,
\begin{equation}
\label{eq:extended_fhf12}
\begin{cases}
\displaystyle\frac{\partial u_\eps(X,t)}{\partial t} = \Delta_X u_\eps(X,t)
&\hbox{ in } \mathbb{R}_+^{m+1}\times (0, \infty),\\
u_\eps(x,0,t) = u_0(x,t) &\hbox{ in }\R^m \times (-\infty,0],\\
\displaystyle \lim_{y\to 0^+}\frac{\partial u_\eps(X,t)}{\partial y} =-\frac{1}{\eps^2}(1-\vert u_\eps\vert^2) u_\eps &\hbox{ in }\R^m \times (0,+\infty).
\end{cases}
\end{equation}

The following result is our main theorem. 

\begin{theor}\label{main} For any given $u_0\in {\dot{H}}^\frac12(\R^m, N)$, the following statements hold:
\begin{itemize} 
\item [\rm{A)}] There exists a global solution $u\in L^\infty(\R_+, {\dot{H}}^\frac12(\R^m, N))$
of the equation of $\frac12$-harmonic map heat flow:
\begin{equation}\label{half-flow}
\begin{cases}(\partial_t -\Delta)^\frac12 u\perp T_u N & \ {\rm{in}}\ \R^{m}\times (0,\infty),\\
u\big|_{t\le 0}=u_0 & \ {\rm{in}}\ \R^{m}.
\end{cases}
\end{equation}
Furthermore, there exists a closed subset $\Sigma\subset\R^m\times (0,\infty)$, with locally finite $\textcolor{red}{m}$-dimensional 
parabolic Hausdorff measure, such that
$u\in C^\infty(\R^m\times (0,\infty)\setminus\Sigma)$, and
\item [\rm{B)}] there exists $T_0>0$, depending on $\|u_0\|_{\dot{H}^\frac12(\R^m)}$, such that $\Sigma\cap (\R^m\times [T_0,\infty))=\emptyset$ and
$$\|\nabla u(\cdot, t)\|_{L^\infty(\R^m)}\le \frac{C}{\sqrt{t}}, \qquad \forall t\ge T_0.$$
Hence there exists a point $p\in N$ such that $u(\cdot, t)\to p$ in $C^2_{\rm{loc}}(\R^m)$ as $t\to\infty$, and
\item [{\rm{C)}}] for any $0<t<T_0$, $\Sigma_t=\Sigma\cap (\R^m\times\{t\})$ has finite $(m-1)$-dimensional Hausdorff measure.
\end{itemize}
\end{theor}

At the end of this section, we would like to remark that when $\frac12\not=s\in (0,1)$,
while Lemma 3.1 for the energy monotonicity inequality remains true, 
the arguments presented in Lemma 4.3 (for the $\epsilon_0$-regularity) and 
in Proposition 5.1 (for uniform boundary $C^{1,\alpha}$-estimates) do not seem to be valid because of the
degeneracy of coefficient function $y^{1-2s}$ in the extended equation. Thus Theorem 1.1 remains 
open for $s\in (0,1)\setminus\{\frac12\}$.

\section{Existence of weak solutions}
\label{sec:weak_existence}

In this section we prove the existence of a weak solution 
of 
\begin{equation}
\label{eq:fhf_1}
\begin{cases}
\displaystyle\left(\partial_t -\Delta\right)^s u \perp T_u \mathbb S^{\ell-1}\,\,\,\,\hbox{ in } \R^m \times (0,\infty),\\
\\
u(x,t) = u_0(x)\,\,\,\,\hbox{ in }\R^m\times (-\infty,0],
\end{cases}
\end{equation}
for any $s \in (0,1)$, here $u_0\in \dot{H}^s(\R^m, \mathbb S^{\ell-1})$. This equation is a mere generalization of \eqref{eq:fhf}, and thanks to \cite{NS}, \cite{ST} fits well in our framework (see also \cite{AudritoTerracini} for a similar setup and related results). It is important to remark that the case $s=1/2$ and $m=1$ corresponds to a geometric problem since the image by those maps are minimal surfaces with free boundary. See \cite{MS}. We will then consider only the case $s=1/2$ in any dimension in the subsequent sections. 
However, we provide here the existence of weak solutions (but not their partial regularity) for the general system \eqref{eq:fhf} 
for all $0<s<1$ when the initial datum $u_0$ is a function of $x$ only.

Here $\left(\partial_t-\Delta\right)^s u$ is defined by the Poisson representation formula (found independently by Nystr\"om-Sande \cite{NS} and by Stinga-Torrea \cite{ST}): For any $u$ belonging to a suitable class of functions (see \cite{NS,ST})
\begin{equation}\label{poisson}
\left(\partial_t-\Delta\right)^s u(x,t)=\int_0^\infty\int_{\R^m} (u(x,t)-u(x-z,t-\tau))K_s(z,\tau)\,dzd\tau,
\end{equation}
where the kernel $K_s$ is given by
$$
K_s(z,\tau)=\frac{1}{(4\pi)^{\frac{m}2}|\Gamma(-s)|} \frac{e^{-\frac{|z|^2}{4\tau}}}{\tau^{\frac{m}{2}+1+s}},\ \forall z\in\R^{m}, \tau>0,
$$ 
where $\Gamma$ denotes the Gamma function.


As in \cite{chen} and in \cite{chen-struwe}, we relax the constraint $u\in \mathbb S^{\ell-1}$ and  introduce the  Ginzburg-Landau type approximation. 
For any $\eps>0$, we consider the problem ($c_s$ is a normalization constant that will be defined later) 
\begin{equation}
\label{eq:approx_fhf}
\begin{cases}
\displaystyle\left(\partial_t  - \Delta \right)^s u_\eps = \frac{c_s}{\eps^2}(1-\vert u_\eps\vert^2) u_\eps \,\,\,\,\hbox{ in }\R^m \times (0,+\infty),\\
\\
u_\eps(x,t)= u_0(x)\,\,\,\,\hbox{ in }\R^m\times (-\infty,0].
\end{cases}
\end{equation}
Here $\displaystyle c_{s}=\frac{\Gamma(1-s)}{2^{2s-1}\Gamma(s)}$.

The proof of the existence of a solution to the approximate problem \eqref{eq:approx_fhf} and of its convergence to a solution of \eqref{eq:fhf_1} heavily relies on the possibility of reformulating the nonlocal problems \eqref{eq:fhf_1} and \eqref{eq:approx_fhf} as local problems but in an extended variable setting (see \cite{NS} and \cite{ST}).

First we recall the extension method for the nonlocal operator $\left(\partial_t-\Delta \right)^s$, then we prove the existence of a solution of the Ginzburg-Landau approximation \eqref{eq:approx_fhf}. 
Finally, we address the problem of the convergence when $\eps\to 0$.

\subsection{Extension method}
\label{ssec:extension_method}
In this subsection we briefly recall the extension method of \cite{NS} and \cite{ST}.
If $u=u(x,t)$ is a function belonging
\footnote{Note that in the papers \cite{NS} and \cite{ST} it is actually considered a slightly different definition for $D(H^s)$ that prescribes that its elements belong to $L^2(\R^{m+1})$. The reason for considering the ''homogeneous'' version \eqref{eq:Hs} lies in the fact that we have to deal with maps satisfying the constraint $\abs{v}=1$ in the whole $\R^{m+1}$.} to 
\begin{equation}
\label{eq:Hs}
D(H^s):=\left\{v\in \mathcal{S}'(\mathbb{R}^{m+1}): \,\hat{v}\in L^1_{\rm{loc}}(\R^{m+1}), \,\,(\xi,\sigma)\mapsto \left((2\pi\vert\xi\vert)^2 + 2\pi i \sigma\right)^s \hat{v}(\xi,\sigma)\in L^2(\mathbb{R}^{m+1})\right\},
\end{equation}
where $\mathcal{S}'(\mathbb{R}^{m+1})$ is the space of tempered distributions and $\hat{v}$ is the Fourier transform with respect to $(x,t)$, 
then we can consider the degenerate parabolic problem in the extended variables $(X,t):=(x,y,t)\in \mathbb{R}^m\times (0,+\infty)\times \mathbb{R}$:
\begin{equation}
\label{eq:extended}
\begin{cases}
\displaystyle y^{\a}\frac{\partial U(X,t)}{\partial t} = \hbox{div}_{X}\big(y^{\a}\nabla_{X}U(X,t)\big)\,\,\,
\hbox{ in } \mathbb{R}_+^{m+1}\times \R, \\
\\
\displaystyle U(x,0,t) = u(x,t),\,\,\,\,\,\hbox{ in } \R^m\times \R.
\end{cases}
\end{equation}
Given the boundary datum $u$ in the regularity class {$D(H^s)$} above, there exists a smooth solution $U$ of the parabolic problem above. 
Moreover,  there holds (see \cite{NS} and \cite{ST})
\begin{equation}
\label{eq:local_to_global}
-\frac{1}{c_{s}}\lim_{y\to 0^+}y^{\a}\frac{\partial U(X,t)}{\partial y} = \left(\partial_t -\Delta\right)^s u.
\end{equation}
The limit in \eqref{eq:local_to_global} is understood in the $L^2(\mathbb{R}^{m}\times\R)$ sense. 
See also \cite{rueland}.

With this discussion in mind we rewrite the nonlocal and nonlinear system \eqref{eq:fhf_1} as the following local and degenerate parabolic problem with nonlinear boundary conditions in the extended variables $(X,t)\in \R^{m+1}_+\times \R$:
\begin{equation}
\label{eq:extended_fhf}
\begin{cases}
\displaystyle y^{\a}\frac{\partial U(X,t)}{\partial t} = \hbox{div}_{X}\left(y^{\a}\nabla_{X}U(X,t)\right),\,\,\,
&\hbox{ in } \mathbb{R}_+^{m+1}\times \R,\\
\\
U(x,0,t) = u_0(x), \,\,\,\,\,&\hbox{ in }\R^m \times (-\infty,0],\\
\\
\displaystyle\lim_{y\to 0^+}y^{\a}\frac{\partial U(X,t)}{\partial y} \perp T_u \mathbb S^{\ell-1}, \,\,\,\,\,\,&\hbox{ on }\mathbb{R}^m\times (0,+\infty),
\end{cases}
\end{equation}
where the limit in the last condition is understood in the $L^2$ sense. 
We note that the previous system for the case $s=1/2$ arises as the harmonic map flow with a free boundary and has been investigated in \cite{chen-lin}. 

Notice that  our solution $u$ to \eqref{eq:fhf_1} is $\S^{\ell-1}$ valued, and therefore  it is not in $L^2(\R^m)$. Nevertheless, one can interpret distributional solutions of  \eqref{eq:fhf_1} directly through traces of weak solutions of \eqref{eq:extended_fhf},  which are defined below. In particular, in \cite{NS,ST} the domain $D(H^s)$ is designed so that the R.H.S. of \eqref{poisson} makes sense. As previously mentioned, we slightly modify this domain to take into account the constraint. In any case, we always interpret solutions of \eqref{eq:fhf_1} via its extension.

\begin{remark}
We also want to point out that this is the flow of harmonic maps with free boundary from a manifold with edge-singularities into the sphere. Indeed, for $a>-1$, the operator $y^{2-a}\text{div}(y^a \nabla)$ is an edge-operator in the sense of \cite{mazzeo-edge}. Therefore, the flow \eqref{eq:extended_fhf} is the Ginzburg-Landau approximation of the heat flow of harmonic maps of a manifold with edge-singularities into the round sphere. See also \cite{jesse} for related results. We postpone a deeper investigation of such flows on singular manifolds to future work. 
\end{remark}

\begin{remark}
We would like to point out that the approach used in \cite{Audrito} would be an alternative way to build weak solutions for our system too. 
\end{remark}

Now we discuss the weak formulation of \eqref{eq:extended_fhf}.
First of all, we introduce some functional spaces. 
Given an open set $A\subset \R^{m+1}_+$, we introduce the Lebesgue and the Sobolev spaces with weights
\begin{equation}
\label{eq:L2m}
L^2(A; y^{1-2s}\d X):=\left\{V:A\to \R^{\ell}: \int_{A}  \vert  V\vert^2 y^{\a}\d X <+\infty \right\},
\end{equation}
and
\begin{equation}
\label{eq:H1m}
H^1(A;y^{\a}\d X):= \left\{V:A\to \R^{\ell}: V \hbox{ and } \nabla_X V\in L^2(A, y^{1-2s}\d X)  \right\},
\end{equation}
endowed with the norm 
\begin{equation}
\label{eq:normH1m}
\norm{V}_{H^1(A; y^{1-2s}\d X)}:= \left(\int_{A} \vert  V\vert^2 y^{\a}\d X + \int_{A}\abs{\nabla_X V}^2 y^{\a}\d X\right)^\frac12.
\end{equation}
Moreover, we let
\begin{equation}
\label{eq:X}
\mathbb{X}^{2s}(A):=\left\{V:A\to \R^{\ell}: \nabla_X V \in L^2(A,y^{\a}\d X)\right\},
\end{equation}
endowed with the semi-norm 
\begin{equation}
\label{eq:normX}
\| V\|_{\mathbb{X}^{2s}(A)}: = \Big( \int_{A} y^{\a} \vert \nabla_X V\vert^2 \d X\Big)^{1/2}.
\end{equation} 
Thanks to \cite[Theorem 2.8]{Nek},  there exists a unique bounded linear operator (the trace operator) 
\begin{equation}
\label{eq:trace}
\text{Tr}: \mathbb{X}^{2s}(\R^{m+1}_+)\to \dot{H}^s(\R^m),
\end{equation}
such that $\text{Tr}V := V\big|_{\R^m\times \left\{0\right\}}$ for any $V\in C^1_{c}(\R^{m+1})$

Finally, given a Banach space $\mathcal{X}$ with norm $\norm{\cdot}_{\mathcal{X}}$,
 we let $L^p(a,b; \mathcal{X})$ ($p\in [1,+\infty]$) denote the 
space of classes of functions which are strongly measurable on $[a,b]$ and with values in 
$\mathcal{X}$ and such that 
\[
\norm{v}_{L^p(a,b;\mathcal{X})}<+\infty,
\]
where 
\[
\norm{v}_{L^p(a,b;\mathcal{X})}:= 
\begin{cases}
\left(\int_{a}^b \norm{v(t)}_{X}^p\d t\right)^{1/p} \qquad &\textrm{ if } \quad p\in [1,+\infty)\\
\textrm{ess sup}_{t\in (a,b)}\norm{v(t)}_{\mathcal{X}}\qquad &\textrm{ if }\quad p=+\infty.
\end{cases}
\]
Moreover, we let 
\[
H^1(a,b;\mathcal{X}):= \left\{v\in L^2(a,b;\mathcal{X}): \quad \frac{\d}{\d t}v \in L^2(a,b;\mathcal{X})\right\},
\]
where the derivative is understood in the sense of distributions (see, e.g., \cite[Chapter 1]{lionsmag})

\begin{defn}
\label{def:weak_sol_extension} Given a $u_0\in \dot{H}^s(\R^m, \mathbb S^{\ell-1})$,
a map $U:\R^{m+1}_+\times \mathbb{R}\to \mathbb{R}^{\ell}$, with {$\abs{U(x,0,t)} =1$ for almost every $(x,t)\in \R^m\times\R$}, is weak solution of \eqref{eq:extended_fhf} if  
\begin{eqnarray}
 & \partial_t U\in L^2(\mathbb R_+; L^2(\R^{m+1}_+,y^{\a}\d X)),\label{eq:weak_reg1}\\
 &U\in L^{\infty}(\mathbb{R}_+;\mathbb{X}^{2s}(\R^{m+1}_+)),\label{eq:weak_reg2}\\
  & U(x,0,t) = u_0(x)\,\,\,\,\,\, \hbox{ a.e. } (x,t)\in \R^m\times (-\infty,0],\label{eq:constraint}  
  \end{eqnarray}
  and
\begin{equation}
\label{eq:weak1}
\int_{0}^{\infty}\int_{\R^{m+1}_+}\left(\langle \partial_t U,\Phi\rangle 
+ \langle \nabla_{X}U,\nabla_{X}\Phi\rangle\right) y^{1-2s}\d X\d t = 0,
\end{equation}
 for any $\Phi \in L^{\infty}(\mathbb R_+;\mathbb{X}^{2s}(\R^{m+1}_+))\cap L^\infty\left(\R_+;L^\infty( \R^{m+1}_+\right))$ with $\Phi(x,0,t)\in T_{U(x,0,t)}\mathbb{S}^{\ell-1}$ 
 for almost every $(x,t)\in \R^m\times (0,+\infty)$. 
 \end{defn}
 
Note that if $U$ is a weak solution according to the above definition taking $\Phi$ with $\Phi(x,0,t)=0$ for almost any $(x,t)\in \R^m\times (0,+\infty)$ we get that $U$ verifies 
 \begin{equation}
 \label{eq:interior}
\displaystyle y^{\a}\frac{\partial U(X,t)}{\partial t} = \hbox{div}_{X}\left(y^{\a}\nabla_{X}U(X,t)\right),\,\,\,\hbox{ in } \mathbb{R}_+^{m+1}\times \R.
 \end{equation}

Owing to the previous definition, we now define what we mean by a weak solution of the original system \eqref{eq:fhf_1}: 
 
 \begin{defn}
\label{def:weak_sol}
Given $u_0\in \dot H^s(\R^m,\mathbb S^{\ell-1})$, 
we say that $u:\R^m\times \R\to \mathbb{S}^{\ell-1}$ is a weak solution of \eqref{eq:fhf_1} if the pair $(U,u)$ with $u=\text{Tr}(U)$ is a weak solution of the extended equation according to Definition \ref{def:weak_sol_extension}. 
\end{defn}

\begin{remark}
It would be possible of course to have a more straightforward definition of weak solutions for  \eqref{eq:fhf_1} by defining suitable function spaces so that the Fourier-Laplace multiplier $(\partial_t-\Delta)^s$ is well defined. This would actually introduce some additional technicalities which are unnecessary for our purposes and we do not pursue along this line. We refer the reader to \cite{rueland} for a related construction.   
\end{remark}

Following \cite{chen}, \cite{chen-struwe},  \cite{bookLW}, in the next Lemma we exploit the symmetry of the constraint 
$\mathbb{S}^{\ell-1}$ to write \eqref{eq:weak1} in an equivalent way that is more suited for the treatment of the nonlinear boundary condition in the limit procedure. 
The reformulation of \eqref{eq:weak1} makes use of test functions defined in $\R^{m+1}_+\times \R$ with values in $\extp_k(\R^{\ell})$. Therefore we have to introduce some notation. 
The exterior algebra of $\mathbb{R}^{\ell}$ is denoted by
 $\extp(\mathbb{R}^{\ell})$ and the exterior (or wedge) product by $\wedge$. If $e_1,\ldots,e_{\ell}$ is the canonical orthonormal basis of $\mathbb{R}^{\ell}$, we let $\extp_{k}(\mathbb{R}^{\ell})$ ($k\le \ell$) be the space of $k$-vectors, namely the subspace of $\extp(\mathbb{R}^{\ell})$ spanned by $e_{i_1}\wedge\ldots\wedge e_{i_k}$ with $(1\le i_1\le \ldots,i_k\le \ell$). We let $\langle \cdot, \cdot\rangle$ denote the scalar product in $\mathbb{R}^{\ell}$. We denote with the same symbol the induced scalar product in $\extp_{k}(\mathbb{R}^{\ell})$
 \begin{equation}
 \label{eq:scalar_prod}
 \langle v_1\wedge\ldots\wedge v_k,w_1\wedge\ldots\wedge w_k\rangle:= \text{det}\left(\langle v_i,w_i\rangle\right),
 \end{equation}
 where $v_i, w_i\in \mathbb{R}^{\ell}$ for $i=1,\ldots,k$.

We finally introduce the Hodge star operator 
\[
\star: \extp_{k}(\mathbb{R}^{\ell})\to \extp_{\ell-k}(\mathbb{R}^{\ell})\qquad 0\le k\le \ell,
\]
by 
\[
\star\left(e_{i_1}\wedge\ldots\wedge e_{i_k}\right):= e_{j_1}\wedge\ldots\wedge e_{j_{\ell-k}},
\]
where $j_1,\ldots,j_{\ell-k}$ is chosen in such a way that $e_{i_1},\ldots,e_{i_p},e_{j_1},\ldots,e_{j_{\ell-k}}$ is a (positive) basis of $\mathbb{R}^{\ell}$.  
The following hold
\begin{eqnarray*}
& \star(1) = e_{1}\wedge\ldots\wedge e_{\ell},\\
&  \star(e_{1}\wedge\ldots\wedge e_{\ell})=1,\\
& \star \star v = (-1)^{k(\ell-k)}v,\,\,\,\,\,\forall v\in \extp_{k}(\mathbb{R}^{\ell}),
\end{eqnarray*}
and
 \begin{equation}
 \label{eq:hodge}
 u\wedge \star v = \langle u,v\rangle e_1\wedge\ldots\wedge e_{\ell},\qquad \hbox{ for any } u,v\in
  \extp_{k}(\mathbb{R}^{\ell}).
 \end{equation}
 or, equivalently,
 \begin{equation}
 \label{eq:hodge2}
 \star\left(u\wedge \star v\right) = \langle u, v\rangle\qquad \hbox{ for any } u,v\in \extp_{k}(\mathbb{R}^{\ell}).
 \end{equation}
In the familiar case in which $u,v$ are vectors in $\mathbb{R}^3$, then the relation above with $\ell = 3$ and $k=1$ gives
\[
\star \left(u\wedge  v\right) = u\times v.
\]
Then, we introduce some new function space.
We set 
\[
\mathbb{X}^{2s}\left(\R^{m+1}_+;\extp_{\ell-2}(\mathbb{R}^{\ell})\right):= \left\{V:\R^{m+1}_+\to \extp_{k}(\mathbb{R}^\ell): \,\,\nabla_X V \in L^2\left(\R^{m+1}_+,y^{1-2s}\d X\right)\right\}.
\]
We have the following
\begin{lemma}
\label{lemma:equivalent_weak}
$U$ is a weak solution in the sense of Definition \ref{def:weak_sol_extension} if and only if $U$ verifies \eqref{eq:weak_reg1}, \eqref{eq:weak_reg2}, \eqref{eq:constraint}, {\color{red}\eqref{eq:interior}} and 
\begin{eqnarray}
\label{eq:weak2}
\displaystyle\int_{0}^{\infty}\int_{\R^{m+1}_+}\left(\langle \partial_t U,\star\left(U\wedge \Psi\right)\rangle +
\langle \nabla_{X}U,\star \left(U\wedge \nabla_{X}\Psi\right)\rangle\right) y^{1-2s}\d X\d t = 0,
\end{eqnarray}
 for any $\Psi\in L^\infty\left(\R_+;\mathbb{X}^{2s}\left(\R^{m+1}_+;\extp_{\ell-2}(\mathbb{R}^{\ell})\right)\right)\cap
  L^\infty\left(\R_+;L^\infty\left(\R^{m+1}_+;\extp_{\ell-2}(\mathbb{R}^{\ell})\right)\right)
 $.
\end{lemma}
\begin{proof}
If $U$ is a weak solution in the sense of Definition \ref{def:weak_sol_extension}, 
then we take $\Phi = \star\left(U\wedge \Psi\right)$ where $\Psi \in L^\infty\left(\R_+;\mathbb{X}^{2s}\left(\R^{m+1}_+;\extp_{\ell-2}(\mathbb{R}^{\ell})\right)\right)$.
Thanks to the properties of the wedge product and of the Hodge-star operator, it is immediate to check that $\Phi$ is indeed a vector field. The fact that $\Phi\in \mathbb X^{2s}(\mathbb R^{m+1}_+)$ for a.e. $t$ is a consequence of the the fact that its components are product of 
functions which lie in $\mathbb X^{2s}(\mathbb R^{m+1}_+)$ and in $L^\infty$ for almost every $t$. We have to check that 
 $\Phi(x,0,t)\in T_{U(x,0,t)}\mathbb{S}^{\ell-1}$, namely that, denoting with $u(\cdot, \cdot):= U(\cdot, 0,\cdot)$ (in the sense of traces), 
\[
\langle u, \star\left(u\wedge \Psi\right)\rangle =0\qquad \hbox{ a.e. in } \R^m\times \R.
\] 
This is a consequence of \eqref{eq:hodge2}. In fact, 
\begin{equation}
\label{eq:ortogonal}
\langle u, \star\left(u\wedge \Psi\right)\rangle = \star\left(
 u \wedge\left(\star\star\left(u\wedge \Psi\right)\right) \right)=(-1)^{\ell-1}\star\left( \left(u\wedge (u\wedge \Psi)\right) \right)= 0.
\end{equation}
Finally, since the Hodge star operator commutes with the covariant differentiation (here derivation in $\R^{m+1}_+$) we have that 
\begin{eqnarray*}
\displaystyle\langle \nabla_X U ,\nabla_X(\star\left(U\wedge \Psi\right))\rangle & = &\langle \nabla_X U, \star\left( \nabla_X U\wedge \Psi\right)\rangle + \langle\nabla_X U, \star\left(U\wedge \nabla_X\Psi\right)\rangle\\
\displaystyle & = &\langle\nabla_X U, \star\left(U\wedge \nabla_X\Psi\right)\rangle,
\end{eqnarray*}
where the first addendum is treated as in \eqref{eq:ortogonal}.
As a result we have that $U$ verifies also \eqref{eq:weak2}. 

On the other hand, let $U$ be a function verifyng \eqref{eq:weak_reg1}, \eqref{eq:weak_reg2}, \eqref{eq:constraint}, \eqref{eq:weak1} and {\color{red}\eqref{eq:interior}}. 
For any given vector field $\Phi$ as in the Definition \ref{def:weak_sol_extension} we set 
\[
\Psi:=\star\left(U\wedge \Phi\right). 
\]
We have that $\Psi \in  L^\infty\left(\R_+;\mathbb{X}^{2s}\left(\R^{m+1}_+;\extp_{\ell-2}(\mathbb{R}^{\ell})\right)\right)$.
Moreover for almost any $(x,t)\in \R^m\times (0,+\infty)$ there holds 
\[
\star\left(U\wedge \Psi\right) = \Phi.
\]
Thus \eqref{eq:interior} and \eqref{eq:weak2} give that 
$U$ verifies also \eqref{eq:weak1} and thus it is a weak solution in the sense of Definition \eqref{def:weak_sol_extension}.
\end{proof}

According to \cite{ST}, given a solution $U$ of the above problem, its trace on $\mathbb{R}^{m}\times \left\{0\right\}$  
\[
u(x,t):=\text{Tr}U(x,y,t),
\]
is indeed a (weak) solution of \eqref{eq:fhf_1}. 
Weak solutions to \eqref{eq:extended_fhf} are constructed as limits of solution of the the (local) extension of the 
Ginzburg Landau approximation \eqref{eq:approx_fhf} of \eqref{eq:fhf_1}. Therefore, for any $\eps>0$ we consider 
the following system
\begin{equation}
\label{eq:approx_ext}
\begin{cases}
\displaystyle y^{\a}\frac{\partial U_\eps(X,t)}{\partial t} = \hbox{div}_{X}\big(y^{\a}\nabla_{X}U_\eps(X,t)\big)\,\,\,
&\hbox{ in } \mathbb{R}_+^{m+1}\times (0, \infty),\\
\\
U_\eps(x,0,t) = u_0(x), \,\,\,\,\,&\hbox{ in }\R^m \times (-\infty,0],\\
\\
\displaystyle\lim_{y\to 0^+}y^{\a}\frac{\partial U_\eps(X,t)}{\partial y} = -\frac{c_{s}}{\eps^2}\left(1-\abs{U_\eps}^2\right) U_\eps, \,\,\,&\hbox{ in } \R^m\times (0,+\infty).
\end{cases}
\end{equation}
%


\subsection{Existence for the approximate problem and a priori estimates}
\label{sec:ex_approx}
In this subsection we discuss the existence of the approximate problem \eqref{eq:approx_ext}.

First of all, we introduce some notation. 
For $\eps>0$ and $V\in \mathbb X^{2s}(\mathbb R^{m+1}_+)$, we introduce the following energy functional
\begin{equation}
\label{eq:energy}
\mathscr{E}_\eps(V, v) := \frac 12\int_{\R^{m+1}_+}y^{\a} \abs{\nabla_X V}^2 \d X +
\frac{c_s}{4\eps^2}\int_{\R^m} (1-\vert v\vert^2)^2 \d x, \ \textrm{for} \ V\in \mathcal{V},
\end{equation}
where $v = \textup{Tr}\,V$. 
We seek for minimizers in the space 
\[
\mathcal{V}:=\Big\{V\in \X: v:=\text{Tr}(V)\,\,|\,\, (|v|^2-1)^2\in L^1(\R^m)\Big\}.
\]

We let $U_0:\R^{m+1}_+\times (-\infty,0]\to \mathbb{R}^{\ell}$ be the Caffarelli-Silvestre extension of $u_0$, namely, for any $t\in (-\infty, 0]$,
\begin{equation}
\label{eq:U0}
\begin{cases}
\displaystyle-\text{div}\left(y^{\a}\nabla U_0(X,t)\right) = 0\qquad &\text{ in }\quad \R^{m+1}_+,\\
\displaystyle U_0(x,0,t) = u_0(x)\qquad &\text{ on }\quad \R^m.
\end{cases}
\end{equation}
Thanks to \cite{CS} we have that the extension operator
\begin{equation}
\label{eq:CS_ext}
E: v\mapsto V, \  \text{with }\ V \hbox{ the unique solution of }\eqref{eq:U0} \hbox{ with boundary datum } v,
\end{equation}
is an isometry from $\dot{H}^s(\R^m)$ to $\mathbb{X}^{2s}(\R^{m+1}_+)$ and we have 
\begin{equation}
\label{eq:trace2}
\norm{\text{Tr}V}_{\dot{H}^{s}(\R^m)} = 
\norm{E(\text{Tr}(V))}_{\mathbb{X}^{2s}(\R^{m+1}_+)}\le 
\norm{V}_{\mathbb{X}^{2s}(\R^{m+1}_+)},
\end{equation}
for any $V\in \mathbb{X}^{2s}(\R^{m+1}_+)$.

Towards the construction of a solution to \eqref{eq:approx_ext} we observe that,  since $U_0$ is constant with respect to time,
the function $U_0$ verifies
\begin{equation}
\begin{cases}
\displaystyle y^{\a}\frac{\partial U_0(X,t)}{\partial t} = \hbox{div}_{X}\big(y^{\a}\nabla_{X}U_0(X,t)\big),\,\,\,
&\hbox{ in } \mathbb{R}_+^{m+1}\times (-\infty,0],\\
\\
\displaystyle U_0(x,0,t) = u_0(x), \,\,\,\,\,&\hbox{ in }\R^m \times (-\infty,0].
\end{cases}
\end{equation}
Therefore, we study existence of a solution of the following initial and boundary value problem:
\begin{equation}
\label{eq:approx_ext_+}
\begin{cases}
\displaystyle y^{\a}\frac{\partial U_\eps(X,t)}{\partial t} = \hbox{div}_{X}\big(y^{\a}\nabla_{X}U_\eps(X,t)\big)\,\,\,
&\hbox{ in } \mathbb{R}_+^{m+1}\times (0,+\infty),\\
\\
U_\eps(x,y,0) = U_0(x,y), \,\,\,\,\,&\hbox{ in }\R^{m+1}_+\times \left\{0	\right\},\\
\\
\displaystyle\lim_{y\to 0^+}y^{\a}\frac{\partial U_\eps(X,t)}{\partial y} = -\frac{c_{s}}{\eps^2}\left(1-\abs{U_\eps}^2\right) U_\eps, \,\,\,&\hbox{ in } \R^m\times (0,+\infty).
\end{cases}
\end{equation}
As a result, if we let $\tilde{U}_\eps$ be a solution of the above problem, then 
\begin{equation}
\label{eq:approx_solution}
U_\eps(X,t):= 
\begin{cases}
\tilde{U}_\eps(X,t)\,\,\,\,\,\,&\hbox{ for }(X, t)\in \mathbb R^{m+1}_+\times (0,\infty),\\
U_0(X,t)\,\,\,\,\,\,\,\,\,\,\,\,\,\,&\hbox{ for }(X, t)\in \mathbb R^{m+1}_+\times (-\infty,0],
\end{cases}
\end{equation}
is a solution of \eqref{eq:approx_ext}.
As the behavior of $t\le 0$ of $U_\eps$ is ruled by $U_0$ which only depends on the known ``initial" condition $u_0$, with some abuse of notation we will use the same symbol $U_\eps$ to denote both a solution of \eqref{eq:approx_ext} and a solution of \eqref{eq:approx_ext_+}. 

We concentrate on \eqref{eq:approx_ext_+}. Since for the moment we work at fixed $\eps$, we do not indicate the dependence on $\eps$ in the notation. 
Existence of a solution can be proven, for instance, by using 
a time discretization scheme. More precisely, for $n\in \mathbb{N}$ we set $\tau:=\frac{T}{n}$ and $t^{k}:=\tau k$ for $k=0,\ldots,n$. 
We set $U^0:=U_0$ and we (iteratively) let $U^k$ (with $k=1,\ldots,n$) be the solution of 
\begin{equation}
\label{eq:time_discrete}
\begin{cases}
U^k - \tau y^{-(1-2s)}\text{div}\left(y^{\a}\nabla_{X}U^k\right) =U^{k-1} , \,\,\,\,\,&\hbox{ in }\R^{m+1}_+,\\
\\
\displaystyle\lim_{y\to 0^+}y^{\a}\frac{\partial U^k}{\partial y} = -\frac{c_{s}}{\eps^2}\left(1-\abs{U^k}^2\right) U^k \,\,\,\,\,\,&\hbox{ in }\R^m\times \left\{0\right\}.
\end{cases}
\end{equation}
Equation \eqref{eq:time_discrete} is the Euler-Lagrange equation for the minimizer of the energy (as in \eqref{eq:energy} we indicate with $u$ the trace of $U$ on $\mathbb{R}^m\times \left\{0\right\}$)
\[
F(U,u):= \frac{1}{2}\int_{\R^{m+1}_+}\frac{\abs{U-U^{k-1}}^2}{\tau}y^{\a}\d X +  \mathscr{E}_\eps(U,u).
\]
Existence of a minimizer in the space $\mathcal{V}$ is a consequence of the Direct method of Calculus of Variations. 
Once we have constructed the discrete solutions $U^k$ for $k=1,\ldots,n$, we can standardly introduce the piecewise constant and piecewise affine (in time) interpolants of the discrete solutions and pass to limit when $\tau$ (the time step) tends to $0$.
This limit procedure gives (we restore the $\eps$-dependence) a solution $U_\eps$ of \eqref{eq:approx_ext_+}. We let $u_\eps(\cdot,\cdot):=U_\eps(\cdot, 0,\cdot)$ (in the sense of traces).  
The function $U_\eps$ 
satisfies by construction the following a priori estimate (that correspond with testing \eqref{eq:approx_ext_+} with $\frac{\partial U_\eps}{\partial t}$ and integrating on $\R^{m+1}_+$)
\begin{eqnarray}
\label{eq:energy_est1}
\int_{\R^{m+1}_+}y^{\a}\abs{\frac{\partial U_\eps}{\partial t}}^2 \d X 
+ \frac{\d}{\d t} \mathscr{E}_\eps(U_\eps,u_\eps)(t)=0.
\end{eqnarray}
Thus, integrating with respect to time in $(0, {T})$ for $0<{T}<\infty$, we get (recall that 
$\textup{Tr}(U_0(\cdot, \cdot, 0)) = u_0(\cdot$ ) and that $u_0\in \mathbb S^{l-1}$ 
a.e. in $\R^m$)  
\begin{eqnarray}
\label{eq:energy_est2}
\int_{0}^{T}\int_{\R^{m+1}_+}y^{\a}\abs{ \frac{\partial U_\eps (X,t)}{\partial t}}^2 \d X \d t
+\int_{\R^{m+1}_+}y^{\a} \abs{ \nabla_X U_\eps (X,t)}^2 \d X\nonumber\\
 +
\frac{c_s}{4\eps^2}\int_{\R^m} (1-\vert u_\eps\vert^2)^2 \d x =
 \int_{\R^{m+1}_+}y^{\a} \vert \nabla_X U_\eps (X,0)\vert^2 \d X
 \le \norm{u_0}_{\dot{H}^s(\R^m)}^2.
\end{eqnarray}
Thus, we obtain 
\begin{eqnarray}  
\label{eq:energy_est}
&&\int_0^{T}\int_{\R^{m+1}_+}y^{\a}\abs{\frac{\partial U_\eps (X,t)}{\partial t}}^2 \d X \d t
+\int_{\R^{m+1}_+}y^{\a} \abs{\nabla_X U_\eps (X,t)}^2 \d X\nonumber
\\
&&\qquad\qquad \qquad\qquad\qquad\qquad\qquad\qquad+\frac{c_s}{4\eps^2}\int_{\R^m} (1-\vert u_\eps\vert^2)^2 \d x \le C, 
\end{eqnarray}
where the constant $C$ does not depend on $\eps$. Thus, 
we conclude that $\partial_t U_\eps$ and $U_\eps$ are uniformly bounded with respect to $\eps$ in the spaces 
\begin{equation}
\label{eq:energy_bound1}
L^2(\mathbb R_+; L^2(\R^{m+1}_+,y^{\a}\d X)) \,\,\textrm{ and }\,\, L^\infty(\mathbb R_+;\mathbb{X}^{2s}(\R^{m+1}_+)),
\end{equation}
respectively.
Moreover, recalling \eqref{eq:approx_solution} we have indeed constructed a solution (still denoted with $U_\eps$) of \eqref{eq:approx_ext_+} that satisfies
\begin{equation}
\label{eq:energy_bound}
\norm{U_\eps}_{H^1(\mathbb R_+; L^2(\R^{m+1}_+,y^{\a}\d X))} + \norm{U_\eps}_{L^\infty(\R_+;\mathbb{X}^{2s}(\R^{m+1}_+))}\le C.
\end{equation}
Note that $u_\eps$ is a solution of \eqref{eq:approx_fhf}.


\subsection{Limit procedure and existence of a weak solution}
\label{ssec:weak_sol}
The energy estimate \eqref{eq:energy_est} and weak compactness results 
guarantee the existence of a map 
$U:\R^{m+1}_+\times \R_+\to \R^{\ell}$ with 

\[
\partial_t U\in L^2\left(\mathbb R_+;L^2(\R^{m+1}_+,y^{\a}\d X)\right)\,\,\textrm{and }\,\,U\in  L^{\infty}\left(\R_+;\mathbb{X}^{2s}(\R^{m+1}_+)\right)
\]

and of a subsequence 
of $\eps$ (not relabeled)
such that 
\begin{eqnarray}
\label{eq:weak_conv1}
\partial_t U_\eps \xrightarrow{\eps\to 0} \partial_t U \qquad &\hbox{ weakly in } L^2\left(\mathbb R_+; L^2(\R^{m+1}_+,y^{\a}\d X)\right),\\
\nabla_X U_\eps \xrightarrow{\eps \to 0} \nabla_X U \qquad&\hbox{ weakly star in } L^{\infty}\left(\R_+;L^2(\R^{m+1}_+,y^{\a}\d X)\right).\label{eq:weak_conv2}
\end{eqnarray}
Moreover, the Aubin-Lions compactness Lemma gives that 
\begin{equation}
\label{eq:strong}
U_\eps \xrightarrow{\eps\to 0} U \hbox{ strongly in } 
L^2_{\text{loc}}\left(\mathbb R_+;L^2_{\text{loc}}(\R^{m+1}_+,y^{\a}\d X)\right).
\end{equation}
Now, denoting with $u$ and with $u_\eps$ the traces of $U$ and of $U_\eps$ on $\R^m\times \left\{0\right\}$ respectively,
we have that 
\begin{equation}
\label{eq:strong_traces}
u_\eps \xrightarrow{\eps \to 0} u\qquad \text{ in } L^2_{\text{loc}}(\mathbb R_+;L^2_{\text{loc}}(\R^m)),
\end{equation}
and thus $u_\eps\to u$ almost everywhere in 
$\R^m\times \mathbb R_+$, up to the extraction of a further subsequence. 
The convergence almost everywhere above combined with the fact that, thanks to estimate \eqref{eq:energy_est},
\[
\lim_{\eps\to 0}\int_{\R^m}\left(1-\vert u_\eps\vert^2\right)^2 \d x = 0,
\]
allows us to reach that $\abs{u(x,t)}=1$ for almost any 
$(x,t)\in \R^m\times \R_+$.
To conclude that $U$ is a weak solution of \eqref{eq:extended_fhf} in the sense of Definition \ref{def:weak_sol_extension} we have to prove that $U$ verifies \eqref{eq:weak2}. 
We consider $\Psi\in L^\infty\left(\R_+;\mathbb{X}^{2s}\left(\R^{m+1}_+;\extp_{\ell-2}(\mathbb{R}^{\ell})\right)\right)\cap L^\infty\left(\R_+;L^{\infty}\left(\R^{m+1}_+;\extp_{\ell-2}(\mathbb{R}^{\ell})\right)\right)$ and we test 
\eqref{eq:approx_ext_+} with 
 $\star\left( U_\eps\wedge \Psi\right)$. 
 For almost any 
 $(x,t)\in \R^m\times (0,+\infty)$
\begin{eqnarray*}
\bigg\langle\frac{1}{\eps^2}\left(1-\abs{u_\eps}^2\right)u_\eps, \star\left(u_\eps\wedge\Psi\right) \bigg\rangle
&= &\star\left(\frac{1}{\eps^2}\left(1-\abs{u_\eps}^2\right)u_\eps\wedge \star\star(u_\eps\wedge \Psi)\right)\\
&= &(-1)^{\ell-1}\star \left(\frac{1}{\eps^2}\left(1-\abs{u_\eps}^2\right)u_\eps\wedge (u_\eps \wedge \Psi)\right)
=0,
\end{eqnarray*}
thanks to \eqref{eq:hodge2} (recall \eqref{eq:ortogonal}). 
For $t\le 0$, we have that $U_\eps(x,0,t) = u_\eps(x,t) = u_0(x)$ and therefore, since $\abs{u_0}=1$ by hypothesis, we conclude that
\[
\frac{1}{\eps^2}\left(1-\abs{u_\eps(x,t)}^2\right)u_\eps(x,t) = 0,\,\,\,\hbox{ for a.e. } x\in \R^m \,\,\,\hbox{ and }t\le 0. 
\]
Thus, 
after integration by parts in space we conclude that $U_\eps$ verifies 
\begin{eqnarray}
\label{eq:weakUeps}
\displaystyle\int_{\mathbb R_+}
\int_{\R^{m+1}_+}\left(\langle \partial_t U_\eps,\star\left(U_\eps\wedge \Psi\right)\rangle +
\langle \nabla_{X}U_\eps,\star \left(U_\eps\wedge \nabla_{X}\Psi\right)\rangle\right) y^{1-2s}\d X\d t = 0.
\end{eqnarray}
Convergences \eqref{eq:weak_conv1}-\eqref{eq:strong} are enough to pass to the limit in equation \eqref{eq:weakUeps} and to obtain 
that $U$ 
verifies \eqref{eq:weak2}. Thus, thanks to Lemma \ref{lemma:equivalent_weak} we conclude that $U$ is indeed a weak solution of \eqref{eq:extended_fhf}. Therefore, the trace of $U$ on $\R^m\times \left\{0\right\}$ is a weak solution of \eqref{eq:fhf_1}.

\section{Monotonicity formula for the approximate problem}

This section is devoted to the derivation of monotonicity formula for \eqref{eq:approx_ext_+}.
For the later purpose, we will provide both global and local versions of such formulas.

For   $t_0\geq 0$  and $0\le R\le \frac{t_0}2$, we  set 
\begin{align*}T_R^+(t_0):=\{ (X,t)\in\R^{m+1}_+\times\R_+:  t_0-4R^2< t< t_0-R^2 \}, 
\quad T_1^+:=T_1^+(0),\\ \partial^+ T_R^+(t_0):=\{ (x,0,t)\in\R^m\times\{0\}\times\R_+:t_0 -4R^2<t<t_0-R^2 \} ,\quad \partial^+ T_1^+:=\partial T_1^+(0).\end{align*}
 For $X_0=(x_0,0)\in\R^m\times\{0\}$ and $0<s<1$, let 
 $$\mathcal{G}_{X_0,t_0}^s(X,t):=\frac{1}{\Gamma(s)(4\pi)^\frac m2 |t-t_0|^{\frac m2 +1-s}} e^{-\frac{|X-X_0|^2}{4|t-t_0|}}, \ t<t_0$$ be the backward fundamental solution of \eqref{eq:approx_ext_+}. 
 For $X_0=0$ and $t_0=0$, we write $\G^s=\G_{X_0,t_0}^s$. Note that 
  $$\nabla \G^s(X,t)=-\frac{X}{2|t|}\G^s(X,t),\ \G^s(RX,R^2t)=R^{-m-2+2s}\G^s(X,t),
  \ \forall(X,t)\in\mathbb R^{m+1}_+\times\mathbb R_{-}, \ R>0.$$

\begin{lemma}\label{lem-mono} For every $Z_0=(X_0,t_0)$ with $X_0\in\partial\R^{m+1}_+$ and $t_0>0$, 
if $U_\eps$ solves \eqref{eq:approx_ext_+} then the following two renormalized energies
\begin{align*}
\mathcal{D}(U_\eps, Z_0, R)&:=  R^2\Big(\frac12\int_{\R^{m+1}_+\times \{t_0-R^2\}}\G_{X_0,t_0}^s(X,t) y^{1-2s}|\nabla U_\eps|^2 dX\\
&\quad+\frac{c_s}{4\eps^2}\int_{\R^{m}\times \{t_0-R^2\}}\G_{X_0,t_0}^s(X,t)(1-|u_\eps|^2)^2dx\Big)
\end{align*}
and
\begin{align*}
\E(U_\eps, Z_0, R)&:=  \frac12\int_{T_R^+(Z_0)}\G_{X_0,t_0}^s(X,t) y^{1-2s}|\nabla U_\eps|^2 dXdt\\
&\quad+\frac{c_s}{4\eps^2}\int_{\partial^+ T_R^+(Z_0)}\G_{X_0,t_0}^s(X,t)(1-|u_\eps|^2)^2dxdt
\end{align*}
are monotone nondecreasing with respect to $R$. Namely,
\begin{equation}\label{mono}
\begin{cases}
\mathcal{D}(U_\ve, Z_0, r)\le \mathcal{D}(U_\ve, Z_0, R), \ 0<r\le R<\sqrt{t_0}, \\  
\E(U_\ve, Z_0, r)\leq \E(U_\ve, Z_0, R), \ \ 0<r\le R< \frac12\sqrt{t_0}.
\end{cases}
\end{equation}
\end{lemma}

\begin{proof} Here we just sketch a proof for $\E(U_\eps, Z_0, R)$. Let us set 
$$U_{\eps,R}(X,t):=U_\eps(RX+X_0,R^2t+t_0),\quad  u_{\eps,R}(x,t):=u_\eps(Rx+x_0,R^2t+t_0)$$
for $X\in\mathbb R^{m+1}_+$ and $t>-R^{-2} t_0$.
Then $U_{\ve,R}$ satisfies 

\begin{equation}
\label{eq:approx_ext-2}
\begin{cases}
y^{\a}\partial_t U_{\eps,R}(X,t)  = \hbox{div}_{X}\big(y^{\a}\nabla_{X}U_{\eps,R}(X,t)\big),
&\hbox{ in } \mathbb{R}^{m+1}_+\times (-R^{-2}t_0, \infty), \\
U_{\eps,R}(x,0,t) = u_{\eps,R}(x,t), &\hbox{ in } \R^m\times (-R^{-2}t_0,\infty),\\
U_{\eps,R}(x,0,t) = u_0(Rx+X_0),&\hbox{ in }\R^m \times (-\infty,-R^{-2}t_0],\\
\displaystyle\lim_{y\to 0^+}y^{\a} \partial_y U_{\eps,R}(X,t)  = -R^{2s}\frac{c_s}{\eps^2}(1-\vert u_{\eps,R}|^2) u_{\eps,R}, &\hbox{ in } \R^m\times (-R^{-2}t_0,\infty).
\end{cases}
\end{equation}

 By the change of variables  $X\mapsto RX+X_0$, $t\mapsto R^2t+t_0$,  we get 
 $$\E(U_\eps, Z_0, R):=  \frac12\int_{T_1^+}\G^sy^{1-2s}|\nabla_X U_{\eps,R}|^2 dXdt+\frac{c_s}{4\eps^2}R^{2s}\int_{\partial^+ T_1^+}\G^s(1-|u_{\eps,R}|^2)^2dxdt. $$ 
Therefore, integrating by parts we obtain 
\begin{align*}& \frac12\frac{d}{d R} \int_{T_1^+}\G^s y^{1-2s}|\nabla_X U_{\eps,R}|^2 dXdt\\
&=\int_{T_1^+}\G^s y^{1-2s}\nabla_X U_{\ve,R}\cdot\nabla_X \partial_R U_{\ve,R}dXdt   \\ 
& =- \int_{T_1^+}\hbox{div}_X[\G^s y^{1-2s}\nabla_X U_{\ve,R}]\cdot \partial_R U_{\ve,R}dXdt 
- \int_{\partial^+ T_1^+}\lim_{y\to 0^+}[\G^s y^{1-2s}\partial_y U_{\ve,R}\cdot \partial_R U_{\ve,R}]dxdt \\
&=- \int_{T_1^+}\G^s y^{1-2s}[ \frac{X}{2t}\cdot  \nabla_X U_{\ve,R}+ \partial_t U_{\ve,R}]\cdot \partial_R U_{\ve,R}dXdt \\ 
&\quad +R^{2s}\frac{c_s}{\ve^2}\int_{\partial^+ T_1^+} \G^s  (1-|u_{\ve,R}|^2)u_{\ve,R}\cdot \partial_R u_{\ve,R}dxdt .
\end{align*}
Here we have used the fact $\frac{X}{2t}=-\frac{X}{2|t|}$, and
that $\partial_R U_{\ve,R}(x,0,t)=\partial_R u_{\ve,R}(x,t)$ for $t>-R^{-2}t_0$, which is a consequence of 
$$\lim_{y\to 0^+} y\partial_y U_\eps (RX, R^2t)=\lim_{y\to0^+}y^{2s}[y^{1-2s}\partial _y U_{\ve}(RX,R^2t)]=0. $$
While
\begin{align*}
&\frac{d}{dR}\left\{\frac{c_s}{4\eps^2}R^{2s}\int_{\partial^+ T_1^+}\G^s(1-|u_{\eps,R}|^2)^2dxdt\right\}\\
&=\frac{sc_s}{2\eps^2}R^{2s-1}\int_{\partial^+ T_1^+}\G^s(1-|u_{\eps,R}|^2)^2dxdt\\
&\quad-R^{2s}\frac{c_s}{\ve^2}\int_{\partial^+ T_1^+} \G^s  (1-|u_{\ve,R}|^2)u_{\ve,R}\cdot \partial_R u_{\ve,R}dxdt.
\end{align*}
Since $$\partial_R U_{\eps,R}=\frac1R(X\cdot\nabla_X U_{\ve,R}+2t\partial_t U_{\ve,R}),$$  we obtain 
  \begin{align}\label{monotonicity1}
\nonumber  \frac{d}{dR}\E(U_\ve, Z_0, R)&=\frac{1}{2R}\int_{T_1^+} \frac{\G^s}{|t|}y^{1-2s}\left|X\cdot\nabla_X U_{\ve,R}(X,t)+2t\partial_t U_{\ve,R}(X,t)\right|^2dXdt \\ &\quad+\frac{sc_s}{2\eps^2}R^{2s-1}\int_{\partial^+ T_1^+}\G^s(1-|u_{\eps,R}|^2)^2dxdt\\&\geq 0.  \nonumber
 \end{align}
This yields the monotonicity of $\E(U_\ve, Z_0, R)$ with respect to $R>0$ and hence
completes the proof. \end{proof}

We will also need the following local energy inequality. Here we denote 
$$P_R^+(Z_0)=B_R^+(X_0)\times [t_0-R^2, t_0+R^2], \ 
\partial^+P_{R}^+(Z_0)=P_{R}^+(Z_0)\cap(\partial\R^{m+1}_+ \times (0,\infty))$$
 for $R>0$ and $Z_0=(X_0,t_0)\in\overline{\R^{m+1}_+}\times \R$.
\begin{lemma}\label{local_energy_ineq1}
If $U_\eps$ solves \eqref{eq:approx_ext_+} then for any $\eta\in C_0^\infty(\mathbb R^{m+1})$ it holds that
\begin{align}\label{local_energy_ineq2}
&\frac{d}{dt}\left\{\int_{\R^{m+1}_+} \frac12|\nabla U_\eps|^2\eta^2+\int_{\R^m} \frac{c_{\frac12}}{4\eps^2} (1-|u_\eps|^2)^2 \eta^2\right\}
+\frac12\int_{\R^{m+1}_+} |\partial_t U_\eps|^2\eta^2\nonumber\\
&\le 4\int_{\R^{m+1}_+} |\nabla U_\eps|^2|\nabla\eta|^2.
\end{align}
In particular, for any $Z_0=(X_0,t_0)\in\overline{\mathbb R^{m+1}_+} \times (0,\infty)$ and $0<R<\frac{\sqrt{t_0}}{2}$, we have that
\begin{equation}\label{local_energy_ineq3}
\int_{P_R^+(Z_0)} |\partial_t U_\eps|^2 \le CR^{-2}\left(\int_{P_{2R}^+(Z_0)}|\nabla U_\eps|^2+\int_{\partial^+P_{2R}^+(Z_0)} 
\frac{c_\frac12}{4\eps^2}(1-|u_\eps|^2)^2\right).
\end{equation}
\end{lemma}
\begin{proof} Multiplying the first equation of \eqref{eq:approx_ext_+} by $\partial_t U_\eps\eta^2$ and integrating the resulting
equation over $\R^{m+1}_+$, and applying the third equation of \eqref{eq:approx_ext_+} in integration by parts, we obtain
\begin{align*}
&\frac{d}{dt}\left\{\int_{\R^{m+1}_+} \frac12|\nabla U_\eps|^2\eta^2+\int_{\R^m} \frac{c_{\frac12}}{4\eps^2} (1-|u_\eps|^2)^2 \eta^2\right\}
+\int_{\R^{m+1}_+} |\partial_t U_\eps|^2\eta^2\nonumber\\
&=-2\int_{\R^{m+1}_+}\langle\eta\partial_t U_\eps, \nabla U_\eps\nabla\eta\rangle
\le \frac12\int_{\R^{m+1}_+} |\partial_t U_\eps|^2\eta^2+2\int_{\R^{m+1}_+}|\nabla U_\eps|^2|\nabla\eta|^2.
\end{align*}
This yields \eqref{local_energy_ineq2}. To see \eqref{local_energy_ineq3},  let $\eta\in C^\infty_0(\R^{m+1})$ be a cut-off function of $B_R(X_0)$, i.e.
$0\le\eta\le 1$, $\eta\equiv 1$ in $B_R(X_0)$, $\eta\equiv 0$ outside $B_{2R}(X_0)$, and $|\nabla\eta|\le 4R^{-1}$.  By Fubini's theorem, there exists $t_*\in (t_0-4R^2, t_0-R^2)$
such that
\begin{align}\label{fubini}
&\int_{B_{2R}^+(X_0)\times\{t_*\}}\frac12|\nabla U_\eps|^2+\int_{(B_{2R}(X_0)\cap\partial \R^{m+1}_+)\times\{t_*\}} \frac{c_{\frac12}}{4\eps^2} \left(1-\abs{u_\eps}^2\right)^2\nonumber\\
&\le \frac{16}{R^2} \left(\int_{P_{2R}^+(Z_0)}\frac12|\nabla U_\eps|^2+\int_{\partial^+P_{2R}^+(Z_0)} \frac{c_{\frac12}}{4\eps^2} \left(1-\abs{u_\eps}^2\right)^2\right).
\end{align}
Now if we integrate \eqref{local_energy_ineq2} for $t_*\le t\le t_0+R^2$ and apply \eqref{fubini}, we obtain that 
\begin{align*} 
\int_{P_R^+(Z_0)} |\partial_t U_\eps|^2&\le \int_{B_{2R}^+(X_0)\times\{t_*\}}\frac12|\nabla U_\eps|^2+\int_{(B_{2R}(X_0)\cap\partial \R^{m+1}_+)\times\{t_*\}} \frac{c_{\frac12}}{4\eps^2} \left(1-\abs{u_\eps}^2\right)^2\\
&\quad+CR^{-2}\int_{P_{2R}^+(Z_0)}|\nabla U_\eps|^2\\
&\le CR^{-2}\left(\int_{P_{2R}^+(Z_0)}|\nabla U_\eps|^2+\int_{\partial^+P_{2R}^+(Z_0)} 
\frac{c_\frac12}{4\eps^2}\left(1-\abs{u_\eps}^2\right)^2\right).
\end{align*} 
This implies \eqref{local_energy_ineq3} and completes the proof.
\end{proof}

\section{$\varepsilon-$regularity result}

From now on, we will always assume $s=\frac12$. Therefore, according to 
\[
c_{s}:= \frac{\Gamma(1-s)}{2^{2s-1}\Gamma(s)},
\]
we have that $c_{1/2}=1$. 

As previously mentioned, we now focus only on the system \eqref{eq:approx_ext_+}. 
We will derive a priori estimates of its solutions $U_\eps$ under a smallness condition
on the renormalized energy.

For  $X_0=(x_0,0)$, $t_0>0$ and $0<R\leq\sqrt{t_0}$,  we set 
$$P_R^+(X_0,t_0):=\left\{(X,t)\in\overline{\mathbb{R}^{m+1}_+}\times(0,\infty) :   |X-X_0|\leq R,\, t_0-R^2\leq t\leq t_0+R^2 \right \}.$$
\begin{lemma} Assume $U_\ve$ is a bounded, smooth solution of \eqref{eq:approx_ext_+}. 
Then $|U_\epsilon|\le 1$
in $\mathbb R^{m+1}_+\times (0, \infty)$.
\end{lemma} 
\begin{proof} We argue by contradiction. Suppose the conclusion were false. Then by the maximum principle
there exists
$Z_0=(x_0, 0, t_0)\in\partial{\mathbb R^{m+1}_+}\times (0, \infty)$ such that 
$$(|U_\ve|^2-1)(Z_0)=\max_{Z\in\overline{\mathbb R^{m+1}_+}\times [0, \infty)}(|U_\ve|^2-1)(Z)>0.$$
Set $\Phi_\ve=|U_\ve|^2-1$. Then it satisfies
\begin{equation*}
\begin{cases} \partial_t \Phi_\ve-\Delta \Phi_\ve=-2|\nabla U_\ve|^2\le 0 & \ {\rm{in}}\ \ \ \mathbb R^{m+1}_+\times [0,\infty),\\
\displaystyle\lim_{y\to 0^+} \frac{\partial\Phi_\ve}{\partial y}(X,t)=\frac{2c_s}{\ve^2}\Phi_\ve(x,0, t)|u_\ve|^2(x,t)
& \ {\rm{on}}\ \ \ \partial\mathbb R^{m+1}_+\times [0,\infty).
\end{cases}
\end{equation*} 
It follows from the Hopf boundary Lemma that 
$$\frac{\partial\Phi_\ve}{\partial y}(Z_0)=\displaystyle\lim_{y\to 0^+} \frac{\partial\Phi_\ve}{\partial y}(x_0,y,t_0)<0.
$$
On the other hand, the boundary condition of $\Phi_\ve$ yields that
$$
\frac{\partial\Phi_\ve}{\partial y}(Z_0)=\frac{2c_s}{\ve^2}\Phi_\ve(Z_0))|u_\ve|^2(Z_0)>0.
$$
We get the desired contradiction.
\end{proof}

The next Lemma is a clearing-out result, which plays a crucial role in the small-energy regularity result. 

\begin{lemma}\label{clearing-out} There exists $\eps_0>0$ such that 
if $U_\ve$ is a  smooth solution of \eqref{eq:approx_ext_+} with $|U_\ve|\leq 1$, that satisfies
$$\E(U_\ve, (X_0,t_0), 1)\leq\ve_0$$
for some $X_0\in\partial\mathbb R^{m+1}_+$ and $t_0>4$, 
then 
$$|U_\ve|\geq\frac12  \ \ {\rm{on}}\ \ P_\delta^+(X_0,t_0)$$ 
for some $\delta>0$ independent of $U_\ve$, $X_0$ and $t_0$.  
\end{lemma}
\begin{proof} We divide the proof into two cases:

\smallskip
\noindent\textbf{Case 1}: $\ve\geq\frac12$. 
Set $$ V_\ve(x,y,t)=\int_0^y U_\ve(x,s,t)ds,\quad y>0.$$ Then $$\partial_y(\partial_t-\Delta) V_\ve(x,y,t) =0,$$ that is, $(\partial_t-\Delta) V_\ve$ is independent of $y\in (0,\infty)$. In particular, we get   $$(\partial_t-\Delta) V_\ve(x,y,t)=(\partial_t-\Delta) V_\ve(x,0,t)=-\partial_y U_\ve(x,0,t)=\frac{c_{\frac12}}{\ve^2}(1-|u_\ve|^2)u_\ve,\quad y>0.$$ 
Note that 
$$\displaystyle\frac{c_{\frac12}}{\ve^2}|1-|u_\ve|^2||u_\ve|\leq4 c_{\frac12} \ {\rm{in}}\ \R^{m+1}_+\times (0,\infty), 
\ {\rm{and}}\ V_\ve=0 \ {\rm{for}}\ y=0.
$$  
Hence, by the standard parabolic theory \cite[Theorem 2.13]{Lieberman} we conclude that $V_\ve$ is bounded   in $C^{2,1}(P_\frac12^+(X_0,t_0))$.    In other words, $U_\ve$ is bounded   in $C^{1,1}(P_\frac12^+(X_0,t_0))$, which gives \begin{align}\label{est-4.1}|U_\ve(X,t)-U_\ve(\tilde X,\tilde t)|\leq C_1(|X-\tilde X|+|t-\tilde t|^\frac12)\quad\text{for } (X,t) \ {\rm{and}}\  (\tilde X,\tilde t)\in P_\frac12^+(X_0,t_0),\end{align} for some $C_1>0$. Now we choose $\delta_1\in (0,\frac12)$ such that   $\delta_1 C_1\leq \frac18$. 
By the monotonicity inequality \eqref{mono} we get that 
\begin{eqnarray*}
&&\frac{c_{\frac12}}{4\ve^2}\int_{\partial\R^{m+1}_+\cap B_{\delta_1}(X_0)}\int_{t_0-4\delta_1^2}^{t_0-\delta_1^2}(1-|u_\ve|^2)^2dxdt\\
&&\leq C\delta_1^{-(m+1)} \frac{c_{\frac12}}{4\ve^2}\int_{\partial\R^{m+1}_+}\int_{t_0-4\delta_1^2}^{t_0-\delta_1^2}
\mathcal{G}_{X_0, t_0}^\frac12(X,t)(1-|u_\ve|^2)^2dxdt\\
\\&&\leq C\delta_1^{-(m+1)}\E(U_\ve,\delta_1, X_0, t_0)\\
&&\leq C\delta_1^{-(m+1)}\E(U_\ve,1, X_0, t_0)\\
&&\leq C\delta_1^{-(m+1)}\ve_0.
\end{eqnarray*}
Therefore, by choosing $\ve_0>0$ sufficiently small we obtain that 
$$|u_\ve|\geq\frac45\quad\text{for }|X-X_0|\leq \delta_1,\, X\in\partial\R^{m+1}_+,\, t_0-4\delta_1^2\leq t\leq t_0-\delta_1^2.$$   From the choice of $\delta_1>0$ we conclude that $|U_\ve|\geq\frac12$ on $P_{\delta_1}^+(X_0,t_0)$, thanks to \eqref{est-4.1}. 

\medskip
\noindent\textbf{Case 2}: $\ve\leq\frac12$. 
Let $X_1=(x_1,y_1)\in B_{\delta}(X_0)$ with $y_1\geq 0$  and $t_1\in(t_0-\delta^2,t_0+\delta^2)$ being fixed. Set $\tilde X_1=(x_1,0)$, and 
 $$\tilde U_\ve(X,t)=U_\ve(\tilde X_1+\ve^2 X, t_1+\ve^4t).$$ Then $\tilde  U_\ve$ satisfies 
\begin{equation}
\label{eq:approx_ext-21}
\begin{cases}
\partial_t \tilde U_{\eps}(X,t)  =\Delta \tilde U_{\eps}(X,t)
&\hbox{ in } \mathbb{R}^{m+1}_+\times \R, \\
\tilde U_{\eps}(x,0,t) =\tilde  u_{\eps}(x,t) &\hbox{ in } \R^m\times (-\ve^{-4}t_1,\infty),\\
\tilde U_{\eps}(x,0,t) =\tilde  u_0(x,t) &\hbox{ in }\R^m \times (-\infty,-\ve^{-4}t_1],\\
\displaystyle\lim_{y\to 0^+} \frac{\partial\tilde  U_{\eps}}{\partial y}(X,t)  = -{c_\frac12} (1-\vert \tilde u_{\eps}|^2)\tilde  u_{\eps}
&\hbox{ in } \R^m\times (-\ve^{-4}t_1, \infty).
\end{cases}
\end{equation}
By the monotonicity inequality \eqref{mono}, we obtain 
$$\E(\tilde U_\ve,(0,0),1)= \E(U_\ve, (\tilde X_1,t_1), \ve^2)\leq \E(U_\ve,(\tilde X_1,t_1), \frac12)
\leq C(\ve_0+\ve_1),$$ where the last inequality follows from  Lemma \ref{lem-conti}. 
 Now one can proceed as in  Case 1 to show that 
 $$|\tilde U_\ve|\geq\frac12\quad\text{for }(X,t)\in P_{\delta_1}^+(0,0),$$ 
 for some $\delta_1>0$ independent of $\ve>0$. In particular, we obtain \begin{align}\label{est-small-y}|U_\ve(X,t)|\geq\frac12\quad\text{for }(X,t)\in P_{\delta_1}^+(X_0,t_0),\,0\leq y\leq \delta_1\ve^2. \end{align} 
 
 Next we find a small $\delta_2>0$ such that $|U_\ve(X,t)|\geq\frac 12$ on $P_{\delta_2}^+(X_0,t_0)$, provided $\ve_0>0$ is sufficiently small. Note that it suffices to consider points $(X,t)\in P_{\delta_2}^+(X_0,t_0)$, $X=(x,y)$ for which $\ve^2\delta_1\leq y\leq\delta_2$, and $\ve>0$  satisfies $\ve^2\leq \frac{\delta_2}{\delta_1}$. 
 To this end we fix an arbitrary   point $(X_1,t_1) \in P_{\delta}^+(X_0,t_0)$ with $X_1=(x_1,y_1),\,  y_1>0$, and    set $R:=\frac14y_1$, $\tilde X_1=(x_1,0)$.

  We claim that for $\delta>0$ small (depending on $\ve_0$) it holds that  
  \begin{align}\label{claim-1}  \int_{t_1-10R^2}^{t_1}\int_{B^+_{10R}(\tilde X_1)}|\nabla U_\ve|^2dXdt\leq C(\ve_0+\ve_1)R^{m+1}, \end{align}  and  \begin{align}\label{claim-2} \frac{1}{\ve^2} \int_{t_1-10R^2}^{t_1}\int_{ |x-x_1|<10R} (1-|u_\ve|^2 )^2dxdt\leq C(\ve_0+\ve_1)R^{m+1}. \end{align} 
To prove the above claims, let us first choose $\delta>0$ small so that we can apply Lemma \ref{lem-conti} with $\ve_1=\ve_0$, $X_1=\tilde X_1 $ and $t_1=t_1+4R^2$. 
Then  by the monotonicity inequality \eqref{mono} and Lemma \ref{lem-conti},
$$\E(U_\ve, (\tilde X_1,t_1+4R^2), 2R)\leq \E(U_\ve, (\tilde X_1,t_1+4R^2),\frac12)\leq C(\ve_0+\ve_1).  $$   
Since  
$$G_{\tilde X_1,t_1+4R^2}(X,t)\sim R^{-(m+1)} \ \ {\rm{for}}\ \ |X-\tilde X_1|\leq 10R
\ \ {\rm{and}}\ \ t_1-10R^2\leq t\leq t_1,$$
\eqref{claim-1} and \eqref{claim-2} follow immediately.  

As  $B_{{\color{red} 4}R}(X_1)\subset B_{10 R}^+(\tilde X_1)$, and   $\nabla_X U_\ve$ satisfies the heat equation on $\R^{m+1}_+\times (0,\infty)$,   by  \eqref{claim-1} we obtain 
(see \cite[page 61]{Evans}) 
\begin{align}\label{est-nabla} |\nabla_X U_\ve (X,t)|\leq \frac{C\sqrt{\ve_0+\ve_1}}{R}\quad\text{for }|X-X_1|\leq 4R
\ \ {\rm{and}}\ \, t_1-9R^2\leq t\leq t_1, \end{align} 
and consequently,  by the standard parabolic estimates,
\begin{align}\label{est-time}   
|\partial_t U_\ve(X,t)|\leq \frac{C\sqrt{\ve_0+\ve_1}}{R^2}
\quad\text{for }|X-X_1|\leq 3R \ \ {\rm{and}}\ \ t_1-8R^2\leq t\leq t_1. 
\end{align}
Setting  $$ \bar U_\ve(t):= \fint_{B_{R}(X_1)}U_\ve(X,t)dX,$$ we see that 
\begin{align} \label{average}
|U_\ve(X,t)-\bar U_\ve(t_1) |\leq C\sqrt{\ve_0+\ve_1} 
\quad\text{for }|X-X_1|\leq R\ \ {\rm{and}}\ \ t_1-8R^2\leq t\leq t_1 , 
\end{align}  thanks to \eqref{est-nabla}-\eqref{est-time}. 

For $\xi\in\R^\ell$ we set $d(\xi)=|1-|\xi||$. Then $d$ is $1$-Lipschitz. Since 
$$d(\bar U_\ve(t_1))\leq d(U_\ve(X,t))+|U_\ve (X,t)-\bar U_\ve(t_1)|, \ \forall X\in B_R(X_1),
\ t_1-R^2\le t\le t_1,$$ 
taking an average integral one gets
\begin{align*} 
&d(\bar U_\ve(t_1)) \\
& \leq  \fint_{t_1-R^2}^{t_1}\fint_{B_{R}(X_1)}   d(U_\ve(X,t))dXdt  + \fint_{t_1-R^2}^{t_1}\fint_{B_{R}(X_1)} |U_\ve (X,t)-\bar U_\ve(t_1) |dXdt \\ 
&\leq \fint_{t_1-R^2}^{t_1}\fint_{B_{R}(X_1)}   d(U_\ve(X,t))dXdt
+C\sqrt{\ve_0+\ve_1},\end{align*} 
thanks to \eqref{average}.  
By Jensen's inequality, and using that $B_R(X_1)\subset B_{5R}^+(\tilde X_1)$, we get \begin{align*} &\left( \fint_{t_1-R^2}^{t_1}\fint_{B_{R}(X_1)}   d(U_\ve(X,t))dXdt\right)^2 \\ 
&\leq \frac{C}{R^{m+1}} \fint_{t_1-R^2}^{t_1}\int_{B^+_{5R}(\tilde X_1)}   d(U_\ve(X,t))^2dXdt  \\
& \leq \frac{C}{R^{m+1}} \fint_{t_1-R^2}^{t_1} \left( R\int_{|x-x_1|<5R} d(u_\ve(x,t))^2dx+R^2\int_{B_{5R}^+(\tilde X_1)}| \nabla d(U_\ve(X,t))|^2dX \right)dt \\ 
&\leq \frac{C}{R^{m+2}}\int_{t_1-R^2}^{t_1}\int_{|x-x_1|<5R}(1-|u_\ve|^2)^2dxdt+\frac{C}{R^{m+1}}\int_{t_1-R^2}^{t_1}\int_{B_{5R}^+(\tilde X_1)} |\nabla U_\ve|^2dXdt\\ &\leq C({\ve_0}+\ve_1)\left(\frac{\ve^2}{R}+1\right).   \end{align*}  
The second inequality above follows from the Poincar\`e inequality,  
and we have used that $d$ is $1$-Lipschitz in the third inequality, and the last inequality follows from the estimates \eqref{claim-1}-\eqref{claim-2}.  

As we have mentioned before, we only need to consider  $4R= y_1\geq  \delta_1\ve^2$ and $\ve^2<\frac{\delta}{\delta_1}$. Thus we obtain 
$$\displaystyle d(\bar U_\ve(t_1))\leq C\sqrt{\ve_0+\ve_1}(1+\frac{1}{\sqrt{\delta_1}}).$$ 
Hence, if $\ve_0>0$ is sufficiently small,  from \eqref{average} we have 
$$|U_\ve(X_1,t_1)|\geq\frac12.$$ 
Consequently, corresponding to this $\ve_0$ we obtain $\delta=\delta_2>0$ as determined by Lemma \ref{lem-conti} for the choice of $\ve_1=\ve_0$. 
\end{proof}

\medskip
Next we show that under a smallness condition on the renormalized energy, $U_\eps$ enjoys a gradient estimate. More precisely, we have

\begin{lemma}\label{eps_gradient_est} There exists $\eps_0>0$,  depending only on $m$, such that if $U_\eps$ is a smooth solution of \eqref{eq:approx_ext_+}, with $|U_\eps|\le 1$, which satisfies, for $Z_0=(X_0, t_0)\in \partial\R^{m+1}_+
\times (0,\infty)$ and some $0<R<\frac{\sqrt{t_0}}2$, 
\begin{equation}\label{small_cond1}
\mathcal{E}(U_\eps, Z_0, R)<\eps_0^2,
\end{equation}
then
\begin{equation}\label{small_est1}
\displaystyle\sup_{P_{\delta_0 R}^+(Z_0)} R^2 |\nabla U_\eps|^2 \le C\delta_0^{-2}
, \ \ \ \displaystyle\sup_{P_{\delta_0 R}^+(Z_0)} R^4|\partial_t U_\eps|^2 \le C\delta_0^{-4},
\end{equation}
where $0<\delta_0<1$ and $C>0$ are independent of $\eps$.
\end{lemma}
\begin{proof} By scaling, we may assume that $t_0>4$ and $R=1$. 
Let $\delta>0$ be as determined by Lemma \ref{clearing-out}. Since $U_\ve$ is smooth in $\overline{\R^{m+1}_+}\times (0,\infty)$,  there exists $\sigma_\ve\in(0,\delta)$ such that 
$$(\delta-\sigma_\ve)^2\max_{P_{\sigma_\ve}^+(Z_0)}\big(|\nabla U_\ve|^2+|\partial_t U_\ve|\big)
=\max_{0\leq\sigma\leq\delta} (\delta-\sigma)^2 \max_{P_{\sigma}^+(Z_0)}\big(|\nabla U_\ve|^2+|\partial_t U_\ve|\big)
.$$ 
Let $Z_1^\eps=(X_1^\ve,t_1^\ve)\in {P_{\sigma_\ve}^+(Z_0)}$ be such that 
$$\max_{P_{\sigma_\ve}^+(Z_0)}\big(|\nabla U_\ve|^2+|\partial_t U_\ve|\big)
=\big(|\nabla U_\ve|^2+|\partial_t U_\ve|\big)(Z_1^\eps):=e_\ve^2.$$ 
Set $\rho_\ve=\frac12(\delta-\sigma_\ve)$. Since $P_{\rho_\eps}^+(Z_1^\eps)\subset P_{\rho_\ve+\sigma_\ve}^+(Z_0)$,  we have that 
\begin{align}
&\max_{P_{\rho_\ve}^+(Z_1^\ve)}(|\nabla U_\ve|^2+|\partial_t U_\ve|)\nonumber\\
&\le \max_{P_{\rho_\ve+\sigma_\ve}^+(Z_0)}(|\nabla U_\ve|^2+|\partial_t U_\ve|)
\leq 4e_\ve^2.   
\end{align}
Write $X_1^\ve=(x_1^\ve, y_1^\ve)\in \R^{m+1}_+$ with $x_1^\ve\in\R^m$ and $y_1^\ve\ge 0$. 
Set $\tilde{X}_1^\ve=(x_1^\ve, 0)\in\partial\R^{m+1}_+$ and
define 
$$\tilde U_\ve(X,t)=U_\ve\big(\tilde X_1^\ve+\frac{X}{{e_\ve}}, t_1^\ve+\frac{t}{e_\ve^2}\big), 
\ (X,t)\in P_{r_\ve}^+(Y_1^\ve, 0),$$ 
where $r_\ve=\rho_\ve{e_\ve}$ and $Y_1^\ve=(0, y_1^\eps {e_\ve})\in\R^{m+1}_+$. 
Setting 
$$\quad \frac{1}{\tilde\ve^{2}}:=\frac{c_{\frac12}}{\ve^{2}{e_\ve} },$$ 
one gets  
\begin{equation} 
\begin{cases}
(\partial_t-\Delta)\tilde U_\ve=0 & \ {\rm{in}} \  P_{r_\ve}^+(Y_1^\ve, 0), \\ 
\partial_\nu \tilde U_\ve=\frac{1}{\tilde\ve^2 } (1-|\tilde u_\ve|^2)\tilde u_\ve & \ {\rm{on}}\  \partial P_{r_\ve}^+(Y_1^\ve, 0)\cap (\partial\R^{m+1}_+\times \R), \\ 
(|\nabla\tilde U_\ve|^2+|\partial_t \tilde U_\ve|)(X,t)\leq 4 & \ {\rm{in}}\   P_{r_\ve}^+(Y_1^\ve, 0), \\ 
(|\nabla\tilde U_\ve|^2+|\partial_t \tilde U_\ve|)(Y_1^\ve,0)|=1. 
\end{cases} 
\end{equation}

Note that if $r_\ve\leq 2$, then from the definition of $\sigma_\ve$ we obtain 
$$\frac{\delta^2}4\sup_{P_\frac\delta 2^+(Z_0)}(|\nabla U_\ve|^2+|\partial_t U_\ve|)\le (\delta-\sigma_\eps)^2\sup_{P_{\sigma_\eps}^+(Z_0)}
(|\nabla U_\ve|^2+|\partial_t U_\ve|)
=4\rho_\eps^2 e_\eps^2=4r_\eps^2\le 16.$$ 
This yields \eqref{small_est1}. 
Thus we may assume $r_\eps>2$ and proceed as follows.
In case that $y_1^\ve{e_\ve}\geq \frac18$,  one  can use the interior regularity of heat equations to conclude that 
\begin{align}\label{Harnack1}
1&=(|\nabla \tilde U_\ve|^2+|\partial_t \tilde U_\ve|)(Y_1^\ve,0)\nonumber\\
&\leq C \big(\|\nabla \tilde U_\ve\|_{L^2(P_\frac19 (Y_1^\ve, 0))}^2+\|\partial_t \tilde U_\ve\|_{L^1(P_\frac1{18} (Y_1^\ve, 0))}\big)
\nonumber\\
&\le C \Big((9e_\ve)^{-(m+1)}\int_{P_{\frac{1}{9e_\ve}}^+(X_1^\eps, t_1^\eps)}|\nabla U_\eps|^2
+(18e_\ve)^{-(m+1)}\int_{P_{\frac{1}{18e_\ve}}^+(X_1^\eps, t_1^\eps)}|\partial_t U_\eps|\Big).
\end{align}
 Next we need

\smallskip
\noindent{\it Claim}. {For $0<r,\sigma<\delta$ with $2r+\sigma<\delta$, it holds
\begin{align}\label{small_energy1}
&(\frac{r}2)^{2-(m+1)}\int_{P_{\frac{r}2}^+(Z_1)}|\partial_t U_\eps|^2+r^{-(m+1)}\Big(\int_{P_{r}^+(Z_1)}\frac12 |\nabla U_\eps|^2+\int_{\partial^+P_{r}^+(Z_1)}\frac{c_{\frac12}}{4\eps^2}(1-|u_\ve|^2)^2\Big)
\nonumber\\
&\le C(\eps_0^2+\eps_1E_0),  
\ \forall Z_1\in P_\sigma^+(Z_0).
\end{align}
Here $\displaystyle E_0=\int_{\mathbb R^{m+1}_+} |\nabla U_0|^2$.}

Assume \eqref{small_energy1} for the moment. Since $\frac{1}{9e_\ve}\le \frac{8y_1^\ve}{9}\le \sigma_\ve$, it follows from \eqref{small_energy1}
and \eqref{Harnack1} that 
$$1\le C(\eps_0^2+\eps_1E_0)+C\sqrt{\eps_0^2+\eps_1E_0},$$  
which is impossible if we choose a sufficiently small $\eps_0$. Therefore we must have $y_1^\ve{e_\ve}\leq\frac18$ and $r_\eps>2$.  
In this case, we see that $\tilde U_\ve$ satisfies \eqref{grad-1}, \eqref{grad-2} and \eqref{grad-3} with $\ve=\tilde\ve$. Hence, by Proposition \ref{prop-regu}, we have that for any $0<\alpha<1$,
$$\|\tilde U_\ve\|_{C^{1+\alpha}(P^+_\frac17(Y_1^\ve, 0))}\leq C(\alpha).$$ 
In particular, for any $0<r<\frac{1}{60}$ it holds that
$$\big|(|\nabla \tilde U_\ve|^2+|\partial_t\tilde U_\ve|)(X,t)-(|\nabla \tilde U_\ve|^2+|\partial_t\tilde U_\ve|)(Y_1^\ve, 0)|\le Cr^\frac14, 
\quad\forall (X,t)\in P_{r}^+(Y_1^\ve,0). $$ 
Since $(|\nabla \tilde U_\ve|^2+|\partial_t\tilde U_\ve|)(Y_1^\ve, 0)=1$, it follows that
$$(|\nabla \tilde U_\ve|^2+|\partial_t\tilde U_\ve|)(X,t)\ge 1-Cr^\frac14, 
\quad\forall (X,t)\in P_{r}^+(Y_1^\ve,0). $$ 
Thus we can find a $0<r_0<\frac{1}{60}$, independent of $\ve$,
such that 
$$(|\nabla \tilde U_\ve|^2+|\partial_t\tilde U_\ve|)(X,t)\ge \frac12,
\quad\forall (X,t)\in P_{r_0}^+(Y_1^\ve,0). $$ 
This, combined with \eqref{small_energy1},  yields that 
\begin{align*}
r_0^2 &\le Cr_0^{-(m+1)}\int_{P_{r_0}^+(Y_1^\ve, 0)}  (|\nabla \tilde U_\ve|^2+|\partial_t\tilde U_\ve|)\\
&=C(\frac{r_0}2\rho_\ve)^{-(m+1)}\int_{P_{\frac{r_0}2\rho_\ve}^+(Z_1^\ve)} (|\nabla U_\ve|^2+|\partial_t U_\ve|)\\
&\le C(\eps_0^2+\eps_1E_0)+C\sqrt{\eps_0^2+\eps_1E_0}.
\end{align*} 
This is again impossible if we choose a sufficiently small $\eps_0$. Therefore we show that $r_\ve=\rho_\ve e_\ve\le 2$ so that  \eqref{small_est1} holds. 

Now we return to the proof of \eqref{small_energy1}.  For simplicity, assume $Z_0=(0,0)$ and write 
$$\G_{X_*, t_*}(X,t)=\G^\frac12_{X_*,t_*}(X,t).$$ 
For any $Z_1=(X_1,t_1)\in P_\sigma^+(0,0)$, the monotonicity formula \eqref{mono} implies
\begin{align}\label{mono4}
&r^{-(m+1)}\left(\int_{P_r^+(Z_1)} \frac12|\nabla U_\eps|^2dXdt+\int_{\partial^+P_r^+(Z_1)} \frac{c_{\frac12}}{4\eps^2} (1-|u_\eps|^2)^2dxdt\right)\nonumber\\
&\le c\left(\int_{T_{r}^+(t_1+2r^2)} \frac12|\nabla U_\eps|^2\G_{X_1,t_1+2r^2} dXdt+\int_{\partial^+T_{r}^+(t_1+2r^2)} \frac{c_{\frac12}}{4\eps^2} (1-|u_\eps|^2)^2\G_{X_1,t_1+2r^2}dxdt\right)\nonumber\\
&\le c\left(\int_{T_{\frac12}^+(t_1+2r^2)} \frac12|\nabla U_\eps|^2\G_{X_1,t_1+2r^2} dXdt+\int_{\partial^+T_{\frac12}^+(t_1+2r^2)} \frac{c_{\frac12}}{4\eps^2} (1-|u_\eps|^2)^2\G_{X_1,t_1+2r^2}dxdt\right)\nonumber\\
&=c\mathcal{E}(U_\eps, (X_1, t_1+2r^2), \frac12)\nonumber\\
&\le C\left(\mathcal{E}(U_\eps, (0,0), 1)+\eps_1E_0\right)\le C(\eps_0^2+\eps_1 E_0),
\end{align}
where we have used  \eqref{mono2} from Lemma \ref{lem-conti} below, 
since $(X_1, t_1+2r^2)\in P_\delta^+(0,0)$. 
Now one can see that \eqref{small_energy1} follows from \eqref{mono4} and \eqref{local_energy_ineq3}.
\end{proof}

\begin{lemma}
\label{lem-conti}Let $\ve_1>0$ be given. Then there exists $\delta=\delta(\ve_1)>0$ such that for every $(X_0,t_0)\in \partial\R^{m+1}_+\times (4,\infty)$, we have 
\begin{align}  \label{mono2}
 \E(U_\ve, (X_1,t_1),\frac12)\leq C(\E(U_\ve, (X_0,t_0),1)+\ve_1 E_0)\quad\text{for every }(X_1,t_1)\in P_\delta^+(X_0,t_0),
 \end{align}
 where $C>0$ is independent of $\eps, \delta$ and $\eps_1$, and $\displaystyle E_0=\int_{\R^{m+1}_+}|\nabla U_0|^2$.
 \end{lemma}
\begin{proof} 
Proof of this Lemma is essentially contained in \cite{chen-struwe} Lemma 2.4. Here we give a sketch of it. 
For $(X_1,t_1)\in P_\delta^+(X_0, t_0)$,   we see that 
\begin{align*} 
G_{X_1,t_1 }(X,t)&\leq \big(\frac{|t-t_0|}{|t-t_1|}\big)^{\frac{m+1}2}
\exp({\frac{|X-X_0|^2}{4|t-t_0|}-\frac{|X-X_1|^2}{4|t-t_1|}}) G_{X_0,t_0}(X,t)\\
&\le C \exp({\frac{|X-X_0|^2}{4|t-t_0|}-\frac{|X-X_1|^2}{4|t-t_1|}}) G_{X_0,t_0}(X,t)\\
&\le C\exp({c\delta^2 \frac{|X-X_0|^2}{4|t-t_0|}}) G_{X_0,t_0}(X,t)\\
&\le \begin{cases} CG_{X_0,t_0}(X,t), & |X-X_0|\le \delta^{-1}\\
C\exp({-c\delta^{-2}}), & |X-X_0|> \delta^{-1}
\end{cases}
\end{align*}
 for any $(X,t)\in T_\frac12^+(t_1)$, where $C, c>0$ are independent of $\delta$ and $\eps$.
 
 Therefore,   for a given $\ve_1>0$ small one can find $\delta>0$  small  so that  
 for  $(X,t)\in T_\frac12^+(t_1)$,  
 \begin{align*}
 G_{X_1,t_1 }(X,t)\leq C\left\{\begin{array}{ll} G_{X_0,t_0}(X,t) &\quad\text{if }  |X-X_0|\leq \delta^{-1},\\ \ve_1&\quad\text{if }|X-X_0|\geq\delta^{-1}.  
 \end{array} \right.   \end{align*}  
 Hence 
 \begin{align*}
&\E(U_\ve,(X_1,t_1), \frac12) \\
 &\leq C\E(U_\ve, (X_0,t_0), 1)+C\ve_1\left( \int_{t_1-1}^{t_1}\int_{\R^{m+1}_+ } |\nabla \tilde U_\ve|^2dXdt+\int_{t_1-1}^{t_1}\int_{\partial\R^{m+1}_+} (1-|\tilde u_\ve|^2)^2dxdt\right) \\
 &\leq C(\ve_0+\ve_1 E_0),  
 \end{align*} 
 thanks to \eqref{eq:energy_est}.  
 \end{proof}

 \section{Boundary $C^{1+\alpha}$-estimate}

 This section is devoted to establishing an uniform boundary $C^{1+\alpha}$-estimates for solutions $U_\eps$ to \eqref{eq:approx_ext_+},
 under the smallness condition of the renormalized energy of $U_\eps$.
 
 Notations: for $r>0$, set 
 $$
 P_r:=B_r\times (-r^2,0], \  P_r^+:=P_r\cap (\R^{m+1}_+\times\R),$$
 and
 $$
 \Gamma_r:=\big\{(x,0,t):|x|<r, \, -r^2<t\leq 0\big\},\  
 \Gamma_r^+:=\big\{(x,y,t):|x|^2+y^2=r^2, \,y>0,\, -r^2<t\leq 0\big\}.
 $$  
 

 Let $\{U_\ve\}_{\ve>0}:P_1^+\to \mathbb R^l$ be a family of solutions to 
 \begin{align}\label{grad-1} \partial_t U_\ve-\Delta U_\ve=0\quad\text{in }\ P_1^+,    \end{align} 
 subject to the Neumann boundary condition 
 \begin{align} \label{grad-2} 
\frac{\partial U_\ve}{\partial\nu} = \frac{1}{\ve^2} (1-|u_\ve|^2)u_\ve\quad\text{on }\ \Gamma_1.
 \end{align}
 Then we have
 \begin{prop} \label{prop-regu}Let $\{U_\ve\}$ be a family of solutions to \eqref{grad-1}-\eqref{grad-2}. Assume that \begin{align}\label{grad-3} 
 \frac{1}{2}\leq |U_\ve|\leq 1,\quad |\partial_t U_\ve|+|\nabla U_\ve|\leq 4\quad\text{in }\ P_1^+. \end{align}
 Then   $\|U_\ve\|_{C^{1+\alpha}(P^+_\frac14)}\leq C(\alpha)$ for every $0<\alpha<1$ and $\ve>0$.  
 \end{prop}
\begin{proof} 
 Write $U_\ve$ in the polar form. Namely, $U_\ve=\rho_\ve\omega_\ve$,
 with $\rho_\ve=|U_\ve|$ and $\omega_\ve=\frac{U_\ve}{\rho_\ve}$. Then $\rho_\ve$ and $\omega_\ve$
 solve
\begin{align} 
\label{grad-4}
\left\{\begin{array}{ll}
\displaystyle\partial_t \rho_\ve-\Delta \rho_\ve=-|\nabla \omega_\ve|^2 \rho_\ve &\quad\text{in }\ \ P_1^+  \\ \displaystyle
\frac{\partial  \rho_\ve}{\partial\nu}=\frac{c_{\frac12}}{\ve^2}(1-|\rho_\ve|^2)\rho_\ve&\quad\text{on }\ \ \Gamma_1 ,
\end{array}\right. \end{align} 
and 
\begin{align} \label{grad-5}
\left\{\begin{array}{ll}
\displaystyle\partial_t \omega_\ve-\Delta\omega_\ve=f_\ve:=
2\frac{\nabla\rho_\ve}{\rho_\ve} \cdot\nabla \omega_\ve+
|\nabla\omega_\ve|^2 \omega_\ve&\quad\text{in }\ \ P_1^+\\ 
\displaystyle\frac{\partial\omega_\ve}{\partial\nu}=0&\quad\text{on }\ \  \Gamma_1 .\end{array}\right. \end{align}  
It follows from \eqref{grad-3} that 
\begin{equation}\label{grad-3'}
\frac12\le\rho_\ve\le 1,\  |\nabla\rho_\ve|\le |\nabla U_\ve|\le 8,
\ {\rm{\and}}\ \big|\nabla\omega_\ve|\le 2\frac{|\nabla U_\ve |}{|U_\ve|}\le 16, 
\ \ {\rm{in}}\ \  P_1^+,
\end{equation}
and hence 
$$|f_\ve|\le 1000 \ \ \ {\rm{in}}\ \ \  P_1^+.$$
From \eqref{grad-5},  we can apply the regularity theory for parabolic equations to conclude that
\begin{align}\label{grad-6}  \displaystyle\big\|\omega_\ve\big\|_{C^{1+\alpha}(P_{\frac78}^+)}\leq C(\alpha) \quad\text{for every }\alpha\in (0,1), \end{align} 
uniformly with respect to $\ve$. 

 Next we show that $\rho_\ve$ is uniformly bounded in $C^{1+\alpha}(P_{\frac14}^+)$. 
 From \eqref{grad-4}, it suffices to prove $C^{\alpha}$-regularity of $\partial_\nu \rho_\ve$ on $\Gamma_\frac12$.     
 To this end, we set $$h_\ve(X,t)=1-\rho_\ve(X,t), \ \ (X,t)\in P_1^+.$$ 
 The boundary condition in \eqref{grad-4},  and   \eqref{grad-3} imply that 
 \begin{align}\label{grad-8} 
 0\leq h_\ve\leq c\eps^2 \quad\text{on }\ \Gamma_1.    
 \end{align} 
 For any fixed $(Y,0)\in \Gamma_{\frac{7}{16}}$,  set  
 $$\h_Y^\ve(X,t):=h_\ve(X+Y,t)-h_\ve(X,t) \quad\text{for }  (X,t)\in P_{\frac{7}{16}}^+.$$ 
 By direct calculations, $\h_Y^\ve$ satisfies
 \begin{align}   
 \label{grad-} 
 \left\{\begin{array}{ll} (\partial_t-\Delta)\h_Y^\ve=\mathbf{f}_Y^\ve &\quad\text{in }\ P_{\frac{7}{16}}^+,\\  
 \displaystyle\frac{\ve^2}{\rho_\ve(1+\rho_\ve)}\partial_\nu\h_Y^\ve+\h_Y^\ve=  \mathbf{g}_Y^\ve
 &\quad\text{on }\ \Gamma_{\frac{7}{16}},
 \end{array}\right.
 \end{align} 
 where 
 $$\mathbf{f}_Y^\ve(X,t)=|\nabla \omega_\ve(X+Y,t)|^2\rho_\ve(X+Y,t)-|\nabla \omega_\ve(X,t)|^2\rho_\ve(X,t),$$
 and  
 $$\mathbf{g}_Y^\ve(X,t)=  \left( 1-\frac{\rho_\ve(X+Y,t)(1+\rho_\ve(X+Y,t))}{\rho_\ve (X,t)(1+\rho_\ve(X,t))} \right) h_\ve(X+Y,t) .$$ 
 From the estimates  \eqref{grad-3}, \eqref{grad-6} and \eqref{grad-8} we infer  
 \begin{align}\label{est-f} 
 |\mathbf{f}_Y^\ve(X,t)|\leq C|Y| ^\alpha\quad\text{in }\ \ P_{\frac{7}{16}}^+, 
 \end{align}   
 and
 \begin{align}\label{est-g} |\mathbf{g}_Y^\ve(X,t)|\leq C\ve^2|Y|^\alpha\quad\text{on }\ \ \Gamma_{\frac{7}{16}}.  
 \end{align}
 From \eqref{grad-3'} we also have
 \begin{equation}\label{est-hY}
 |\mathbf{h}_Y^\ve(X,t)|\le C|Y|^\alpha, \quad\text{in }\ \ P_{\frac{7}{16}}^+.
 \end{equation}
 
 Denote by
 $$m_Y^\ve=\big\|\mathbf f_Y^\ve\big\|_{L^\infty(P_{\frac{7}{16}}^+)},
 \  n_Y^\ve=\big\|\mathbf g_Y^\ve\big\|_{L^\infty(P_{\frac{7}{16}}^+)},
 \ p_Y^\ve =\big\|\mathbf{h}_Y^\ve\big\|_{L^\infty(P_{\frac{7}{16}}^+)}.
 $$
 Now we need
 
\medskip
\noindent{\bf{Claim}}. {\it There exists a function $\phi_Y^\ve \in C^\infty(\overline{B_{\frac13}^+})$
satisfying
\begin{equation}\label{initial-asatz}
\begin{cases}
\h_Y^\ve\le \phi_Y^\ve\le 10 p_Y^\ve\le C|Y|^\alpha & \ {\rm{in}\ } B_{\frac13}^+,\\
\displaystyle\frac{4\ve^2}{3} \partial_\nu \phi_Y^\ve+\phi_Y^\ve=0 & \ {\rm{on}}\ \partial B_{\frac13}^+
\cap\big\{x\in\mathbb R^{m+1}_+: x_{m+1}=0\big\}.
\end{cases}
\end{equation}}

To verify this claim, first choose a $\frac13<r_0<\frac{7}{16}$ such that
$B_{\frac13}^+\subset B_{r_0}^m\times [0, r_0]$. Next define
$\phi_Y^\ve(x)=\psi_Y^\ve(x_{m+1}): B_{r_0}^m\times [0,r_0]\to\mathbb R$,
where $\psi_Y^\ve\in C^\infty([0,r_0])$ satisfies 
$$ \psi_Y^\ve(\ve^2)=10 p_Y^\ve; \ 10 p_Y^\ve\le \psi_Y^\ve(t)\le 20 p_Y^\ve, \ \ve^2< t\le r_0,$$
and 
\begin{equation}\label{psi}
\frac{4\ve^2}{3}(\psi_Y^\ve)'(t)+\psi_Y^\ve(t)=0, \ 0\le t\le \ve^2.
\end{equation} 
Solving \eqref{psi} yields
$$
\psi_Y^\ve(t)=10 e^{\frac34(1-\frac{t}{\ve^2})} p_Y^\ve,  \ 0\le t\le \ve^2.
$$ 
It is clear that \eqref{initial-asatz} holds by restricting $\phi_Y^\ve$ on $B_{\frac13}^+$.
 
   
Finally we let $\widehat{\h}_Y^\ve: B_\frac13^+\to \mathbb R$ be the unique  solution of
the initial and boundary value problem:
\begin{align}\label{auxi}
\left\{\begin{array}{ll}  (\partial_t-\Delta)\widehat{\h}_Y^\ve=m_Y^\ve &\quad\text{in }P_{\frac13}^+,\\ 
\widehat{\h}_Y^\ve = \phi_Y^\ve &\quad\text{on } B_{\frac13}^+\times \{-(\frac13)^2\},\\
\widehat{\h}_Y^\ve=\phi_Y^\ve &\quad\text{on }\Gamma_{\frac13}^+,\\ 
\displaystyle\frac{4\ve^2}3\partial_\nu \widehat{\h}_Y^\ve+\widehat{\h}_Y^\ve=0 
&\quad\text{on }\Gamma_\frac13.  
\end{array}\right.
\end{align} 
For $\widehat{\h}_Y^\ve$, it follows from both the maximum principle and the $C^0$-estimate
(see \cite{Lieberman}) that 
\begin{equation}\label{mp1}
0\le \widehat{\h}_Y^\ve\le C\big(m_Y^\ve+\|\phi_Y^\ve\|_{L^\infty(B_\frac13^+)}\big)\le C|Y|^\alpha
\ \ {\rm{in}}\ \ P_{\frac13}^+.
\end{equation}
Furthermore,  we have the following uniform gradient estimate, whose proof will be given in the Appendix A.
\begin{align} \label{C1-est}
\big\|\widehat{\h}_Y^\ve\big\|_{C^1(P_{\frac{7}{24}}^+)}\leq C\big(m_Y^\ve+\|\phi_Y^\ve\|_{L^\infty(B_\frac13^+)}\big)\le C|Y|^\alpha, \quad\text{for every }\ve>0.  
\end{align} 
This, combined with the boundary condition on $\Gamma_\frac13$  in \eqref{auxi}, 
implies that there exists a constant $C>0$, independent of $\ve$, such that
\begin{align} \label{est-barh} 
0\leq\widehat{\h}_Y^\ve\leq C\ve^2|Y|^\alpha
\quad\text{on }\Gamma_{\frac {7}{24}}.
\end{align}

Define functions $H_{Y,\ve}^+$ and $H_{Y,\ve}^-$ on $P_\frac13^+$ 
by letting 
$$H^\pm_{Y, \ve}(X,t)=\widehat{\h}_Y^\ve(X,t)\pm \h_Y^\ve(X,t)+n_Y^\ve
\quad\text{in }\ P_{\frac13}^+.$$ 
Then it is easy to verify that 
\begin{align}  \label{auxi1}
\left\{  \begin{array}{ll} 
(\partial_t-\Delta)H^\pm_{Y, \ve}\geq0&\quad \text{in }\ P_{\frac13}^+,\\ 
H^\pm_{Y,\ve}\ge 0 & \quad\text{on }\ B_{\frac13}^+\times \{-(\frac13)^2\},\\
H_{Y, \ve}^\pm\geq0&\quad\text{on }\ \Gamma_{\frac13}^+,\\ 
\displaystyle\frac{4\ve^2}{3}\partial_\nu H_{Y,\ve}^\pm+H_{Y,\ve}^\pm\geq0 
&\quad\text{on }\ \Gamma_\frac13. 
\end{array}\right.
\end{align} 
Applying the maximum principle to \eqref{auxi1} (see \cite{Lieberman}), we conclude that
$$H^\pm_{Y, \ve}\geq0\ \ {\rm{in}}\ \ P_{\frac13}^+,$$
or, equivalently,
$$-\widehat{\h}_Y^\ve(X,t)-n_Y^\ve\le \h_Y^\ve(X,t)\le 
\widehat{\h}_Y^\ve(X,t)+n_Y^\ve, \ (X,t)\in P_{\frac13}^+.$$
Hence we obtain that 
$$|\h_Y^\ve(X,t)|\leq \widehat{\h}_Y^\ve(X,t)+n_Y^\ve \le C\ve^2|Y|^\alpha,
\ (X,t)\in \Gamma_{\frac {7}{24}},$$ 
thanks to \eqref{est-f}, \eqref{est-g} and \eqref{est-barh}. 
In other words, we have  
\begin{align}\label{est-Holder-1} 
\frac{1}{\ve^2}|\rho_\ve(X+Y,t)-\rho_\ve(X,t)|\leq C|Y|^\alpha
\quad\text{for }(X,t)\in \Gamma_{\frac {7}{24}},\, |Y|\leq \frac{7}{16}. 
\end{align}

\medskip   
For every fixed $T\in (-\frac{7}{16},0]$,  set
$$\h_T^\ve(X,t)=h_\ve(X,t+T)-h_\ve(X,t),\  (X, t)\in P_{\frac{7}{16}}^+.$$ 
Then  by an argument similar to that for \eqref{est-Holder-1} we can show that 
\begin{align}\label{est-Holder-2} 
\frac{1}{\ve^2}|\rho_\ve(X,t+T)-\rho_\ve(X,t)|\leq C(r)|T|^\frac\alpha2\quad\text{for }(X,t)\in 
\Gamma_{\frac{7}{24}} ,\, -\frac {7}{16}\leq T\leq 0. 
\end{align} 
Combining \eqref{est-Holder-1} and \eqref{est-Holder-2}, and applying
the boundary condition of $\rho_\ve$ in equation \eqref{grad-4},
we conclude that 
$$\big\|\partial_\nu \rho_\ve\big\|_{C^\alpha(\Gamma_\frac{7}{24})}\leq C$$
holds uniformly with respect to $\ve$.

Now Proposition \ref{prop-regu} follows immediately from the standard parabolic regularity theory 
for equation \eqref{grad-4} (see \cite{Lieberman}). \end{proof}
   
\section{Passing to the limit and partial regularity}

This section is devoted to the proof of our main theorem on the existence of partially smooth solutions
of the heat flow of $1/2$-harmonic maps. 

\medskip
\noindent{\bf Completion of proof of Theorem \ref{main}}:
\begin{proof}
For $s=\frac12$, let $\{U_\eps\}_{\eps>0}$ be a family of solutions of \eqref{eq:approx_ext_+}, satisfying
the bound \eqref{eq:energy_est2}, and $U:\R^{m+1}_+\times [0,\infty)\to\R^l$ be the weak limit of $U_\eps$ as
$\eps\to 0$.  It it readily seen that  $U\in C^\infty(\R^{m+1}_+\times (0,\infty))$ solves
\begin{equation}\label{limit_eqn}
\partial_t U-\Delta U=0\ \ {\rm{in}}\ \R^{m+1}_+\times (0,\infty); \  \ U\big|_{t=0}=U_0 \ \ {\rm{on}}\  \R^{m+1}_+,
\end{equation}
and $u:=U\big|_{\partial\R^{m+1}_+\times (0,\infty)}$ is a weak solution
of the equation  of $1/2$-harmonic map heat flow:
$$
(\partial_t-\Delta)^{\frac12}u\perp T_u \mathbb S^{l-1} \ \ {\rm{on}}\ \ \R^m\times (0,\infty); \ \ u\big|_{t=0}=u_0 \ \ {\rm{on}}\ \ \R^m.
$$
We are left with showing $u$ enjoys the partial regularity as stated in Theorem \ref{main}. 
To show this, let $\eps_0>0$ be the constant determined by Lemma \ref{eps_gradient_est} and define the singular set $\Sigma\subset\partial\R^{m+1}_+\times (0,\infty)$ by
\begin{equation}\label{sing_set}
\Sigma=\bigcap_{R>0}\Big\{Z_0=(X_0, t_0)\in \partial\R^{m+1}_+\times (0,\infty): \liminf_{\eps\to 0} \mathcal{E}(U_\eps, Z_0, R)\ge \eps_0^2\Big\}.
\end{equation}
It is well-known that the monotonicity inequality \eqref{mono} implies that $\Sigma$ is a closed set in $\partial\R^{m+1}_+\times (0,\infty)$.
Furthermore, similar to the proof of Lemma \ref{eps_gradient_est} and Lemma \ref{lem-conti},
we have that for any $Z_0=(X_0,t_0)\in\Sigma$, there exists a $0<r_0<\sqrt{t_0}$ such that for all $0<r<r_0$,
$$
r^{-m}\Big(\int_{P_r^+(Z_0)} |\nabla U_\eps|^2+\int_{\partial^+P_r^+(Z_0)} \frac{1}{\eps^2}(1-|u_\eps|^2)^2\Big)\ge c\eps_0^2.
$$
Now we can apply Vitali's covering Lemma  to show that for any compact set $K\subset  \overline{\R^{m+1}_+}\times (0,\infty)$,
the $m$-dimensional Hausdorff measure of $\Sigma\cap K$ is finite, i.e.,
$$
\mathcal{P}^m(\Sigma\cap K)\le C(E_0, K)<\infty.
$$
It follows from the definition of $\Sigma$ that for any $Z_1=(X_1, t_1)\in\partial\R^{m+1}_+\times (0,\infty)\setminus \Sigma$, we can find a radius $0<R_1<\frac{\sqrt{t_1}}{2}$ so that
$$
\liminf_{\eps\to 0} \mathcal{E}(U_\eps, Z_1, R_1)< \eps_0^2.
$$
Hence by Lemma Lemma \ref{eps_gradient_est} and Lemma \ref{lem-conti} we can conclude that there exists a $\delta_1>0$, independent of $\eps$,
such that for any $\alpha\in (0,1)$,
\begin{equation}
\big\|U_\eps\big\|_{C^{1+\alpha}(P_{2\delta_1 R_1}^+(Z_1))}\le C(\eps_0, \alpha).
\end{equation}
Thus $U_\eps\rightarrow U$ in $C^{1+\alpha}(P_{\delta_1 R_1}^+(Z_1))$. In particular, $U\in C^{1+\alpha}_{\rm{loc}}(\overline{\mathbb R^{m+1}_+}\times (0,\infty)\setminus\Sigma)$.
Applying higher order boundary regularity theory of \eqref{limit_eqn}, we conclude that $U\in C^{\infty}_{\rm{loc}}(\overline{\mathbb R^{m+1}_+}\times (0,\infty)\setminus\Sigma)$. This yields part A) of Theorem \ref{main}.

Observe that for any sufficiently large $t_0>0$, and $X_0\in\partial\R^{m+1}_+$, if choose $R=\frac{\sqrt{t_0}}2$, then
\begin{align*}
\mathcal{E}(U_\eps, (X_0,t_0), R)
&= \int_0^{\frac{3t_0}4} \int_{\R^{m+1}_+} \frac12|\nabla U_\eps|^2 \mathcal{G}_{X_0,t_0}dXdt+
\int_0^{\frac{3t_0}4}\int_{\partial \R^{m+1}_+}\frac{c_{\frac12}}{4\eps^2}(1-|u_\eps|^2)^2 \mathcal{G}_{X_0,t_0}dxdt\\
&\le  Ct_0^{-\frac{m+1}2}\int_0^{\frac{3t_0}4} \Big(\int_{\R^{m+1}_+} \frac12|\nabla U_\eps|^2dX
+\int_{\partial \R^{m+1}_+}\frac{c_{\frac12}}{4\eps^2}(1-|u_\eps|^2)^2dx\Big)dt\\
&\le Ct_0^{-\frac{m-1}2}\int_{\R^{m+1}_+} \frac12|\nabla U_0|^2dXdt=Ct_0^{-\frac{m-1}2}E_0<\eps_0^2,
\end{align*}
uniformly in $\eps$, provided $t_0>\big(\frac{CE_0}{\eps_0^2}\big)^{\frac{2}{m-1}}$.
Here we have used \eqref{eq:energy_est2}.  

Hence by Lemma \ref{eps_gradient_est} and Lemma \ref{lem-conti}, we can conclude
that $\Sigma\cap (\partial\R^{m+1}_+\times [t_0,\infty))=\emptyset$, and 
$U_\eps\rightarrow U$ in $C^2_{\rm{loc}}(\overline{\R^{m+1}_+}\times [t_0,\infty))$.
Furthermore,  it holds that
$$
|\nabla U(X,t)|\le \frac{c}{\sqrt{t}},
$$ for all $X\in\overline{\R^{m+1}_+}$ and $t$ sufficiently large. There exists a point $p\in \mathbb S^{l-1}$
such that $U(\cdot, t)\rightarrow p$ in $C^2_{\rm{loc}}(\overline{\R^{m+1}_+})$ as $t\to \infty$. Hence
$u(\cdot, t)$ also converges to $p$ in $C^2_{\rm{loc}}({\partial\R^{m+1}_+})$ as $t\to \infty$. This yields part B) of Theorem \ref{main}.

The proof of part C) can be done in the same way as in Cheng \cite{Cheng}. We sketch it as follows. 
First recall that for any $\delta>0$, there exists a sufficiently large $K(\delta)>0$ such that for any $t_0>0$ and $0<R<\frac{\sqrt{t_0}}2$, it holds 
for $t_0-4R^2\le t\le t_0-R^2$, 
$$
\mathcal{G}_{(X_0,t_0)}(X,t)\le\begin{cases} R^{-(m+1)} & \forall X\in\R^{m+1}_+,\\
\delta \mathcal{G}_{(X_0,t_0)+(0, R^2)}(X, t) & \ {\rm{if}}\ X\in \R^{m+1}_+ \ {\rm{and}}\ |X-X_0|\ge K(\delta) R.
\end{cases}
$$  
Hence 
\begin{align*}
&\E(U_\epsilon, (X_0, t_0), R) \\
&\le R^{-(m+1)}\int_{t_0-4R^2}^{t_0-R^2}\Big(\int_{B_{K(\delta)R}^+(X_0)} |\nabla U_\eps|^2
+\int_{B_{K(\delta)R}^+(X_0)\cap\partial \R^{m+1}_+} \frac{c_\frac12}{4\eps^2}(1-|u_\eps|^2)^2\Big)\,dt\\  
&+\delta \int_{t_0-4R^2}^{t_0-R^2}\Big(\int_{\R^{m+1}_+} |\nabla U_\eps|^2\mathcal{G}_{(X_0, t_0+R^2)}
+\int_{\partial \R^{m+1}_+} \frac{c_\frac12}{4\eps^2}(1-|u_\eps|^2)^2\mathcal{G}_{(X_0, t_0+R^2)}\Big)\,dt.
\end{align*}
On the other hand,
\begin{align*}
&\delta \int_{t_0-4R^2}^{t_0-R^2}\Big(\int_{\R^{m+1}_+} |\nabla U_\eps|^2\mathcal{G}_{(X_0, t_0+R^2)}
+\int_{\partial \R^{m+1}_+} \frac{c_\frac12}{4\eps^2}(1-|u_\eps|^2)^2\mathcal{G}_{(X_0, t_0+R^2)}\Big)\,dt\\
&\le 2\delta \int_{t_0-4R^2}^{t_0-R^2} (R^2+t_0-t)^{-1} \mathcal{D}(U_\eps, (X_0, t_0+R^2), \sqrt{R^2+t_0-t}) \,dt\\
&\le 2\delta \Big(\int_{t_0-4R^2}^{t_0-R^2} (R^2+t_0-t)^{-1}\,dt\Big) \mathcal{D}(U_\eps, (X_0, t_0+R^2), \sqrt{R^2+t_0})\\
&\le C\delta (t_0+R^2)^{\frac{1-m}{2}}E_0\le \frac12\eps_0^2,
\end{align*}
provided $\delta>0$ is chosen to be sufficiently small. Here we have used the monotonicity inequality for
$\mathcal{D}(U_\eps, (X_0, t_0+R^2), r)$ in the proof.

Note that $\Sigma_{t_0}=\displaystyle\cap_{0<\textcolor{red}{R}<\sqrt{t_0}}\Sigma_{t_0}^R$, where
$$
\Sigma_{t_0}^R=\Big\{X_0\in\partial\R^{m+1}_+: \ \liminf_{\eps\to 0} \E(U_\eps, (X_0, t_0), R)\ge\eps_0^2\Big\}.
$$
Thus we obtain that for any $X_0\in \Sigma_{t_0}^R$, it holds
$$
R^{m+1}\le \frac{2}{\eps_0^2} \lim_{\eps\to 0}\int_{t_0-4R^2}^{t_0-R^2}\Big(\int_{B_{K(\delta)R}^+(X_0)} |\nabla U_\eps|^2
+\int_{B_{K(\delta)R}^+(X_0)\cap\partial \R^{m+1}_+} \frac{c_\frac12}{4\eps^2}(1-|u_\eps|^2)^2\Big)\,dt
$$
so that by Vitali's covering Lemma we can show that
$$
H^{m-1}_{K(\delta) R}(\Sigma_{t_0}^R) \le C(K(\delta), E_0).
$$
This implies $H^{m-1}(\Sigma_{t_0})<\infty$, after sending $R\to 0$. 

\end{proof}

\section{Appendix A: Uniform estimate of heat kernels}

In this section, we will sketch a proof of the gradient estimate \eqref{C1-est} for the solution
$\widehat{\h}_Y^\ve$ of the auxiliary equation \eqref{auxi}, which holds uniformly with respect to
$\ve$.  We refer the reader to \cite{Lieberman} Theorem 4.31, in which
an estimate similar to \eqref{C1-est} is established but with a constant possibly depending on $\ve$.
Here we will provide a proof based on an explicit Green function representation
of the heat equation under an oblique boundary condition. 

First recall the heat kernel in $\R^{m+1}$ given by
$$
\Gamma(x,t)=\begin{cases} \displaystyle\frac{1}{(4\pi t)^{\frac{m+1}2}} \exp \big(-\frac{|x|^2}{4t}\big), &
(x,t)\in \mathbb R^{m+1}\times (0,\infty),\\
0, & (x,t)\in \mathbb R^{m+1}\times (-\infty, 0].
\end{cases}
$$
For $y=(y_1,\cdots, y_m, y_{m+1})\in \mathbb R^{m+1}_+$, denote 
$y^*=(y_1,\cdots, y_m, -y_{m+1})$. Define $G^\ve(x,y,t): \mathbb R^{m+1}_+\times
\mathbb R^{m+1}_+\times \mathbb R\to\mathbb R$ by 
\begin{equation}\label{green1}
G^\ve(x,y, t)= \Gamma(x-y, t)-\Gamma(x-y^*, t)-2\int_0^\infty e^{-\frac{3}{4\ve^2}\tau}
D_{m+1}\Gamma(x-y^*+\tau e_{m+1}, t)\,d\tau,
\end{equation}
where $\displaystyle D_{m+1}\Gamma(z,t)=\frac{\partial \Gamma}{\partial x_{m+1}} (z,t)$ and
$e_{m+1}=(0', 1)\in\mathbb R^{m+1}$. Then we have
\begin{lemma} $G^\ve$ is the Green function of the heat equation in
$\mathbb R^{m+1}_+$ with an oblique boundary condition: for any fixed $y\in \mathbb R^{m+1}_+$, 
\begin{equation}\label{heat-oblique}
\begin{cases}
(\partial_t-\Delta) G^\ve(x,y,t)= \delta(x-y)\delta(t), & (x,t)\in \mathbb
 R^{m+1}_+\times \mathbb R_+,\\
\displaystyle\frac{\partial G^\ve}{\partial x_{m+1}}(x,y, t)-\frac{3}{4\ve^2} G^\ve(x,y,t)=0, & x\in \partial\mathbb R^{m+1}_+\times
[0,\infty).
\end{cases}
\end{equation}
\end{lemma}
\begin{proof} Since $y^*\in \mathbb R^{m+1}_{-}$ for $y\in\mathbb R^{m+1}_+$, it follows that
$x-y^*\not=0$ and $x-y^*+\tau e_{m+1}\not =0$ for any $x\in\mathbb R^{m+1}_+$ and $\tau>0$.
Hence we have
$$(\partial_t-\Delta) G^\ve(x,y,t)=(\partial_t-\Delta) \Gamma(x-y, t)=\delta(x-y)\delta(t).$$
To check the boundary condition, let $x\in\partial\mathbb R^{m+1}_+$. Then we have
that $x_{m+1}=0$ and $|x-y|=|x-y^*|$ so that $\Gamma(x-y,t)=\Gamma(x-y^*, t)$ and
$D_{m+1}\Gamma(x-y^*,t)=-D_{m+1}\Gamma(x-y,t)$. Hence
\begin{eqnarray*}
&&\displaystyle\frac{\partial G^\ve}{\partial x_{m+1}}(x,y, t)-\frac{3}{4\ve^2} G^\ve(x,y,t)\\
&&=-2D_{m+1}\Gamma(x-y^*,t)-2\int_0^\infty e^{-\frac{3}{4\ve^2}\tau}
\frac{\partial}{\partial x_{m+1}}[D_{m+1}\Gamma(x-y^*+\tau e_{m+1}, t)]\,d\tau\\ 
&&+\frac{6}{4\ve^2}\int_0^\infty e^{-\frac{3}{4\ve^2}\tau}D_{m+1}\Gamma(x-y^*+\tau e_{m+1}, t)\,d\tau\\
&&=-2D_{m+1}\Gamma(x-y^*,t)-2\int_0^\infty \frac{\partial}{\partial\tau}
\big(e^{-\frac{3}{4\ve^2}\tau}D_{m+1}\Gamma(x-y^*+\tau e_{m+1}, t)\big)\,d\tau\\
&&=0
\end{eqnarray*}
holds for $x\in\partial\mathbb R^{m+1}_+$.
\end{proof}

For any bounded $f\in C^\infty(\mathbb R^{m+1}_+\times [0,\infty))$, it is well-known that
the unique smooth solution of
\begin{equation}
\begin{cases}(\partial_t-\Delta) u=f & {\rm{in}}\ \mathbb R^{m+1}_+\times [0,\infty),\\
\displaystyle\frac{\partial u}{\partial x_{m+1}}-\frac{3}{4\ve^2} u=0 & {\rm{on}}\ \partial\mathbb R^{m+1}_+\times [0,\infty),\\
u=0 & {\rm{on}}\ \mathbb R^{m+1}_+\times \{0\},
\end{cases}
\end{equation}
is given by the Duhamel formula
\begin{equation}\label{duhamel1}
u_\ve(x,t)=\int_0^t\int_{\mathbb R^{m+1}_+} G^\ve(x, y, t-s) f(y,s)\,dyds, 
\ (x,t)\in \mathbb R^{m+1}_+\times [0,\infty).
\end{equation}

Now we are ready with the proof of the following theorem.

\begin{theor} \label{schauder} For any $f\in C^\infty(\overline{\mathbb R^{m+1}_+\times [0,\infty)})$ and $\ve>0$,
let $u^\ve:\mathbb R^{m+1}_+\times [0,\infty)\to\mathbb R$ be given by \eqref{duhamel1}.
Then for any $0<\alpha<1$ there exists a constant $C=C(m,\alpha)>0$ such that 
\begin{equation}\label{C2a-est}
\big\|u^\ve\big\|_{C^{2+\alpha}(\mathbb R^{m+1}_+\times [0,\infty))} 
\le C\big\|f\big\|_{C^{\alpha}(\mathbb R^{m+1}_+\times [0,\infty))}, \ \forall \ve>0. 
\end{equation}
\end{theor}
\begin{proof} Decompose $G^\ve$ by $G^\ve =G^\ve_1+G^\ve_2$, where
$$G_1^\ve(x,y,t)=\Gamma(x-y, t)-\Gamma(x-y^*, t);  \ G_2^\ve=G^\ve-G^\ve_1,$$
and write $u^\ve=u_1^\ve+u_2^\ve$, where
$$
u_1^\ve(x,t)=\int_0^t\int_{\mathbb R^{m+1}_+} G^\ve_1(x,y, t-s)f(y,s)\,dyds;
\ u_2^\ve=u^\ve-u_1^\ve.
$$

Since $G_1^\ve(x,y,t)$ is the Green function of the heat equation on $\mathbb R^{m+1}_+$ with zero Dirichlet boundary condition,  by the standard Schauder theory
(see \cite{Lieberman}) we have that  $u_1^\ve\in C^\infty\big(\overline{\mathbb R^{m+1}_+}\times [0,\infty)\big)$ and
$$
\big\|u_1^\ve\big\|_{C^{2+\alpha}\big(\overline{\mathbb R^{m+1}_+}\times [0,\infty)\big)}
\le C(m, \alpha) \big\|f\big\|_{C^{\alpha}\big(\overline{\mathbb R^{m+1}_+}\times [0,\infty)\big)}.
$$

To prove a similar estimate for $u_2^\ve$ we first note that \eqref{green1} gives
\begin{equation}
\label{eq:G2eps}
G_2^\ve(x,y,t)=-2\int_0^\infty e^{-\frac{3}{4\ve^2}\tau}
D_{m+1}\Gamma(x-y^*+\tau e_{m+1}, t)\,d\tau.
\end{equation}
By direct computation we have that 
\[
D_{m+1}\Gamma\left(x-y^*+\tau e_{m+1},t\right) = \frac{-1}{2}\frac{1}{\left(4\pi t\right)^{\frac{m+1}{2}}}\frac{(x_{m+1}+y_{m+1}+\tau)}{t}\textrm{exp}\left(-\dfrac{\abs{x-y^*+\tau e_{m+1}}^2}{4t}\right).
\]
Moreover, by the very definition of $y^*$
\[
\abs{x-y^* + \tau e_{m+1}}^2 = \abs{x' -y'}^2 + (x_{m+1}+y_{m+1}+\tau)^2,
\] 
where we recall
\[
x' := \left(x_1,\ldots,x_m,0\right)\qquad y'=\left(y_1,\ldots,y_m,0\right).
\]
Therefore \eqref{eq:G2eps} becomes 
\[
G_2^{\eps}(x,y,t) = \frac{1}{\left(4\pi t\right)^{\frac{m+1}{2}}}e^{-\abs{x'-y'}^2}
\int_{0}^{+\infty}\frac{(x_{m+1}+y_{m+1}+\tau)}{t} \textrm{exp}\left(-\frac{3\tau}{4\eps^2}-\dfrac{\abs{x_{m+1}+y_{m+1}+\tau}^2}{4t}\right)\d \tau.
\]
We change variables in the integral according to 
\[
r:= \frac{x_{m+1}+y_{m+1}+\tau}{\sqrt{t}}.
\]
Moreover, we write 
\[
-\abs{x'-y'}^2 = -\abs{x-y^*}^2 + (x_{m+1}+y_{m+1})^2.
\]
Then,
\[
G_2^{\eps}(x,y,t) = \frac{1}{\left(4\pi t\right)^{\frac{m+1}{2}}}e^{-\abs{x-y^*}^2}e^{(x_{m+1}+y_{m+1})^2 +\frac{3}{4\eps^2}(x_{m+1}+y_{m+1})}
\int_{\frac{x_{m+1}+y_{m+1}}{\sqrt{t}}}^{+\infty}r \textrm{exp}\left(-\frac{3r\sqrt{t}}{4\eps^2}-\frac{r^2}{4}\right)\d r.
\]
We introduce the function $\Theta_\eps:[0,+\infty)\to \R$ given by 
\[
\Theta_\eps(\lambda):=e^{(x_{m+1}+y_{m+1})^2 +\frac{3}{4\eps^2}(x_{m+1}+y_{m+1})}
\int_{\lambda}^{+\infty}r \textrm{exp}\left(-\frac{3r\sqrt{t}}{4\eps^2}-\frac{r^2}{4}\right)\d r,
\]
and thus $G_2^{\eps}$ is represented as 
\[
G_2^{\eps}(x,y,t) = \frac{1}{\left(4\pi t\right)^{\frac{m+1}{2}}}e^{-\abs{x-y^*}^2}\Theta_\eps\left(\frac{x_{m+1}+y_{m+1}}{\sqrt{t}}\right).
\]

We have that for any $\eps>0$, $\Theta_\ve\in C^\infty([0,\infty))$.
Moreover, since 
\[
\Theta'_{\eps}(\lambda) = -e^{(x_{m+1}+y_{m+1})^2 +\frac{3}{4\eps^2}(x_{m+1}+y_{m+1})}\lambda \textrm{exp}\left(-\frac{3\lambda\sqrt{t}}{4\eps^2}-\frac{\lambda^2}{4}\right),
\]
we get that 
\[
\Theta_\eps'\left(\frac{x_{m+1}+y_{m+1}}{\sqrt{t}}\right) = -\left(\frac{x_{m+1}+y_{m+1}}{\sqrt{t}}\right)e^{-\frac{(x_{m+1}+y_{m+1})^2 + 4t }{4t}},
\]
which is bounded, uniformly with respect to $\eps$ and with respect to $t>0$.
Therefore, we conclude thanks to 
Schauder theory 
that 
$u_2^\ve\in C^\infty\big(\overline{\mathbb R^{m+1}_+}\times [0,\infty)\big)$ and
$$
\big\|u_2^\ve\big\|_{C^{2+\alpha}\big(\overline{\mathbb R^{m+1}_+}\times [0,\infty)\big)}
\le C(m, \alpha) \big\|f\big\|_{C^{\alpha}\big(\overline{\mathbb R^{m+1}_+}\times [0,\infty)\big)}.
$$
Combining the estimates for $u_1^\ve$ and $u_2^\ve$ yields \eqref{C2a-est}. 

\end{proof}

\medskip
Now we will give a proof of \eqref{C1-est}. To do it, 
let $\eta_1\in C^\infty_0(B_{\frac13}^m\times (-(\frac13)^2,0))$ be such
that $\eta_1=1$ in $B_{\frac{7}{24}}^m\times (-(\frac{7}{24})^2, 0)$, and $\eta_2\in C^\infty_0([0,\infty)$ be such
that $\eta_2=1$ in $[0, \frac13]$ and $\eta_2=0$ in $[\frac23, \infty)$.
Define $\eta(x,t)=\eta_1(x',t)\eta_2(x_{m+1})$ for $(x,t)\in \R^{m+1}_+\times \R$.
Then by direct calculations we obtain that
\begin{eqnarray*}
&&\big(\frac{\partial}{\partial x_{m+1}} (\widehat{\h}_Y^\ve \eta)-\frac{3}{4\ve^2} \widehat{\h}_Y^\ve \eta\big)(x,t)\\
&&=\widehat{\h}_Y^\ve(x,t) \eta_1(x',t)\eta_2'(x_{m+1})
+\big(\frac{\partial}{\partial x_{m+1}} \widehat{\h}_Y^\ve -\frac{3}{4\ve^2} \widehat{\h}_Y^\ve\big)(x,t) \eta(x,t)\\
&&=0+0=0
\end{eqnarray*}
holds for any $(x,t)\in \partial\mathbb R^{m+1}_+\times (0,\infty)\cap \Gamma_{\frac13}^+$.
Hence by Duhamel's formula, we conclude that for any $(x,t)\in P_{\frac{7}{24}}^+$, it holds
\begin{eqnarray}
(\widehat{\h}_Y^\ve \eta)(x,t)&=&\int_{\mathbb R^{m+1}_+\times (0,\infty)}
G^\ve(x,y, t-s) (\partial_t-\Delta) (\widehat{\h}_Y^\ve \eta)(y,s)\,dyds\nonumber\\
&=&m_Y^\ve \int_{\mathbb R^{m+1}_+\times (0,\infty)}
G^\ve(x,y, t-s) \eta(y,s)\,dyds\nonumber\\
&+& \int_{\mathbb R^{m+1}_+\times (0,\infty)}
G^\ve(x,y, t-s)\widehat{\h}_Y^\ve(y,s) (\partial_t\eta-\Delta\eta)(y,s)\,dyds\nonumber\\
&+&2 \int_{\mathbb R^{m+1}_+\times (0,\infty)} (\nabla_y G(x,y,t-s)\nabla\eta(y,s)+G(x,y,t-s)\Delta\eta(y,s))
\widehat{\h}_Y^\ve(y,s)\,dyds\nonumber\\
&=:&A^\ve(x,t)+B^\ve(x,t)+C^\ve(x,t).
\label{green2}
\end{eqnarray}
Applying Theorem  \ref{schauder}, there exists a constant $C>0$ independent of $\ve$ 
such that 
$$
\big\|A^\ve\big\|_{C^{2+\alpha}(\mathbb R^{m+1}_+\times (0,\infty))}\le C m_Y^\ve\le C|Y|^\alpha.
$$
For $B^\ve$ and $C^\ve$, it is not hard to see that
$$
\big\|\nabla B^\ve\big\|_{C^\alpha(\mathbb R^{m+1}_+\times (0,\infty))}
+\big\|\nabla C^\ve\big\|_{C^\alpha(\mathbb R^{m+1}_+\times (0,\infty))}
\le C \big\|\widehat{\h}_Y^\ve\big\|_{C^0(P_\frac13^+)}\le Cp_Y^\ve\le C|Y|^\alpha.
$$
Putting these estimates together, we conclude that
$\widehat{\h}_Y^\ve$ satisfies the gradient estimate \eqref{C1-est}.

\section{Appendix B: Proof of Theorem \ref{main} for general targets}

In this section, we will sketch the modifications that are necessary in order to show Theorem \ref{main} for any compact Riemannian
manifold $N\hookrightarrow \R^l$.  

To do it, first recall that there exists a constant $\delta_N>0$ such that both the nearest point projection map
$$\Pi_N: N_{\delta_N}\equiv \big\{y\in\R^l: \ d(y, N)<\delta_N\}\to N$$
and the square of distance function to $N$, $d^2(p, N)=|p-\Pi_N(p)|^2$,  are smooth in the $\delta_N$-neighborhood of $N$.  

Now let $\chi\in C_0^\infty([0,\infty))$ be such that 
$$\chi(t)=t \ {\rm{for}}\ 0\le t\le \delta_N^2; \ \ \chi(t)=2\delta_0 \ {\rm{for}}\  t\ge (2\delta_N)^2.$$
Then we replace the potential function $\frac{1}{4\eps^2}(1-|u|^2)^2$ by $\frac{1}{\eps^2} \chi(d^2(u, N))$. More precisely, we
consider the following approximated system:
\begin{equation}\label{half-flow1}
\begin{cases}
(\partial_t-\Delta) U_\eps=0 & \ {\rm{in}}\ \R^{m+1}_+\times (0,\infty),\\
U_\eps\big|_{t=0}=U_0 &\ {\rm{on}}\ \R^{m+1}_+,\\
\displaystyle\lim_{y\to 0^+} \frac{\partial U_\eps}{\partial y} =\frac{c_\frac12}{\eps^2} \chi'(d^2(U_\eps, N))D_{U_\eps}d^2(U_\eps, N)
&\ {\rm{on}}\ \R^m\times (0,\infty).
\end{cases}
\end{equation}

As in \eqref{eq:energy_est2}, it is readily seen that any solution $U_\eps$ of \eqref{half-flow1} satisfies the following energy inequality:
\begin{align}
\label{eq:energy_est3}
&\int_{0}^t\int_{\R^{m+1}_+}\abs{ \frac{\partial U_\eps (X,r)}{\partial t}}^2 \d X \d r
+\int_{\R^{m+1}_+}\abs{ \nabla_X U_\eps (X,t)}^2 \d X\nonumber\\
&+
\frac{c_\frac12}{\eps^2}\int_{\R^m} \chi(d^2(u_\eps, N)) \d x =
 \int_{\R^{m+1}_+} \vert \nabla_X U_0 (X)\vert^2 \d X
 \le \norm{u_0}_{\dot{H}^\frac12(\R^m)}^2.
\end{align}

As in section 3, we can similarly define the renormalized energies $\mathcal{D}(U_\eps, Z_0, R)$
and $\E(U_\eps, Z_0, R)$ for $U_\eps$ by simply replacing the term $(1-|u_\eps|^2)^2$ by 
$\chi(d^2(u_\eps, N))$. For example, 
\begin{align*}
\E(U_\eps, Z_0, R)&:=  \frac12\int_{T_R^+(Z_0)}\G_{X_0,t_0}(X,t) |\nabla U_\eps|^2 dXdt\\
&\quad+\frac{c_\frac12}{\eps^2}\int_{\partial^+ T_R^+(Z_0)}\G_{X_0,t_0}(X,t)\chi(d^2(u_\eps, N))dxdt.
\end{align*}
Then by the same argument as in Lemma 3.1, we have
\begin{lemma} For $Z_0=(X_0,t_0)\in \partial \R^{m+1}_+\times (0,\infty)$,
if $U_\eps$ solves \eqref{half-flow1} then it holds
that
\begin{align*}
\mathcal{D}(U_\eps, Z_0, r)\le \mathcal{D}(U_\eps, Z_0, R), \ \forall 0<r\le R<\sqrt{t_0},\\
\E(U_\eps, Z_0, r)\le \E(U_\eps, Z_0, R), \ \forall 0<r\le R<\frac{\sqrt{t_0}}{2}.
\end{align*}
\end{lemma}

As in Lemma 3.2, we also have the local energy inequality.
\begin{lemma} \label{local_energy_ineq5}  For any $\eta\in C_0^\infty(\mathbb R^{m+1})$, if $U_\eps$ solves \eqref{half-flow1} then 
it holds that
\begin{align}\label{local_energy_ineq6}
&\frac{d}{dt}\big\{\int_{\R^{m+1}_+} \frac12|\nabla U_\eps|^2\eta^2+\int_{\R^m} \frac{c_{\frac12}}{\eps^2} \chi(d^2(u_\eps, N)) \eta^2\big\}
+\frac12\int_{\R^{m+1}_+} |\partial_t U_\eps|^2\eta^2\nonumber\\
&\le 4\int_{\R^{m+1}_+} |\nabla U_\eps|^2|\nabla\eta|^2.
\end{align}
In particular, for any $Z_0=(X_0,t_0)\in\overline{\mathbb R^{m+1}_+} \times (0,\infty)$ and $0<R<\frac{\sqrt{t_0}}{2}$, we have that
\begin{equation}\label{local_energy_ineq7}
\int_{P_R^+(Z_0)} |\partial_t U_\eps|^2 \le CR^{-2}\Big(\int_{P_{2R}^+(Z_0)}|\nabla U_\eps|^2+\int_{\partial^+P_{2R}^+(Z_0)} 
\frac{c_\frac12}{\eps^2} \chi(d^2(u_\eps, N))\Big).
\end{equation}
\end{lemma}

We also have the following clearing out result for any solution $U_\eps$ of \eqref{half-flow1}. 
\begin{lemma}\label{clean_out1} There exists $\eps_0>0$ such that 
if $U_\eps$ solves \eqref{half-flow1} and satisfies
$$\E(U_\eps, (X_0,t_0), 1) \le\eps_0^2,$$
for some $X_0\in\partial\R^{m+1}_+$ and $t_0>4$, then $d(U_\eps, N)\le\delta_N$ and $\chi(d^2(U_\eps, N))=d^2(U_\eps, N)$
hold on $P_\beta^+(X_0,t_0)$ for some
$\beta>0$ that is independent of $U_\eps, X_0,$ and $t_0$.
\end{lemma}

The next Lemma,  analogous to Proposition 5.1, plays a crucial role in the proof.
\begin{lemma}\label{c1a} Let $\{U_\eps\}_{\eps>0}$ be a family of solutions to \eqref{half-flow1}. Assume that
\begin{equation}\label{unif_bdd}
d(U_\eps, N)\le\delta_N, \ |\partial_t U_\eps|+|\nabla U_\eps|\le 4 \ {\rm{in}}\ P_1^+.
\end{equation}
Then $\|U_\eps\|_{C^{1+\alpha}(P_\frac14^+)}\le C(\alpha)$ for any $\alpha\in (0,1)$ and $\eps>0$.
\end{lemma}

\begin{proof} The proof is similar to that of Proposition 5.1 (see also \cite[pages 342-346]{chen-lin}). Since $U_\eps(P_1^+)\subset N_{\delta_N}$,
we can decompose 
$$
U_\eps=V_\eps+\nu_N(V_\eps) \rho_\eps\  \ {\rm{in}}\ \ P_1^+.
$$
Here $V_\eps=\Pi_N(U_\eps), \ \rho_\eps=d(U_\eps, N)=|U_\eps-V_\eps|$, and $\nu_N(V_\eps)\in (T_{V_\eps}N)^\perp$ is a smooth unit vector field
in the normal space $(T_{V_\eps}N)^\perp$. By direct calculations, we obtain that
\begin{align*}
&0=\partial_t U_\eps-\Delta U_\eps\\
&\ =(\mathbb{I}_{l}+\rho_\eps \nabla_{V_\eps}\nu_N(V_\eps)) (\partial_t V_\eps-\Delta V_\eps)
 +\nu_N(V_\eps)(\partial_t\rho_\eps-\Delta\rho_\eps)\\
&\ \ \ \ -2\nabla(\nu_N(V_\eps))\nabla\rho_\eps-\rho_\eps\nabla^2_{V_\eps}\nu_N(V_\eps)(\nabla V_\eps,\nabla V_\eps)
\end{align*}
hold in $P_1^+$.  If we multiply the equation above by $\nu_N(V_\eps)$ and observe that
$$
\langle (\mathbb{I}_{l}+\rho_\eps \nabla_{V_\eps}\nu_N(V_\eps)) (\partial_t V_\eps-\Delta V_\eps), \nu_N(V_\eps)\rangle
=\langle\Delta V_\eps, \nu_N(V_\eps)\rangle=- \nabla_{V_\eps}\nu_N(V_\eps)(\nabla V_\eps, \nabla V_\eps),
$$ 
we can show that $V_\eps$ and $\rho_\eps$ solve
\begin{equation}\label{polar1}
\begin{cases}\displaystyle
(\mathbb{I}_{l}+\rho_\eps \nabla_{V_\eps}\nu_N(V_\eps)) \big(\partial_t V_\eps-\Delta V_\eps)
=\rho_\eps\Pi_N(V_\eps)\big(\nabla^2_{V_\eps}\nu_N(V_\eps)(\nabla V_\eps,\nabla V_\eps)\big)\\ 
\quad\quad\qquad\ \ \ \  \ \ \ \ \ -2\nabla(\nu_N(V_\eps))\nabla\rho_\eps +\rho_\eps \nabla_{V_\eps}\nu_N(V_\eps)(\nabla V_\eps, \nabla V_\eps)\nu_N(V_\eps) \ & {\rm{in}}\ P_1^+,\\
\displaystyle\frac{\partial V_\eps}{\partial y}=0 \ & {\rm{in}}
\ \Gamma_1.
\end{cases}
\end{equation}
and
\begin{equation}\label{polar2}
\begin{cases}
\displaystyle
\partial_t\rho_\eps-\Delta\rho_\eps=\rho_\eps \langle \nabla^2_{V_\eps}\nu_N(V_\eps)(\nabla V_\eps,\nabla V_\eps), \nu_N(V_\eps)\rangle\\
\qquad\qquad\qquad\quad-\rho_\eps \nabla_{V_\eps}\nu_N(V_\eps)(\nabla V_\eps, \nabla V_\eps)
\ & {\rm{in}}\ P_1^+,\\
\displaystyle\frac{\partial\rho_\eps}{\partial y}=\frac{2c_\frac12}{\eps^2}\rho_\eps \ & {\rm{in}}
\ \Gamma_1.
\end{cases}
\end{equation}
Here we have used the fact that $\nabla_p\rho_\eps(p)=\nu_N(\Pi_N(p))$ for $p\in N_{\delta_N}$,
so that the boundary condition for $U_\eps$ implies that on $\Gamma_1$, 
\begin{align*}
&0=\frac{\partial U_\eps}{\partial y}-\frac{c_\frac12}{\eps^2} \chi'(d^2(U_\eps, N))D_{U_\eps}d^2(U_\eps, N)\\
&\ \ =\frac{\partial V_\eps}{\partial y} +\frac{\partial \nu_N(V_\eps)}{\partial y} \rho_\eps
+\big(\frac{\partial\rho_\eps}{\partial y}-\frac{2c_\frac12}{\eps^2} \rho_\eps\big) \nu_N(V_\eps).
\end{align*}
If we multiply this equation by $\nu_N(V_\eps)$ and observe that
$\langle\frac{\partial V_\eps}{\partial y},   \nu_N(V_\eps)\rangle=\langle\frac{\partial \nu_N(V_\eps)}{\partial y}, \nu_N(V_\eps)\rangle=0$,
we would obtain the above boundary condition for $\rho_\eps$. 
On the other hand, the boundary condition for $V_\eps$ follows from the following identity 
$$
0=\frac{\partial V_\eps}{\partial y} +\frac{\partial \nu_N(V_\eps)}{\partial y} \rho_\eps
=\big(\mathbb I_l+\rho_\eps\nabla_{V_\eps}\nu_N(V_\eps)\big)\frac{\partial V_\eps}{\partial y},
$$
and the invertibility of the map $\big(\mathbb I_l+\rho_\eps\nabla_{V_\eps}\nu_N(V_\eps)\big):\R^l\to\R^l$.

Note that (\ref{unif_bdd}) implies that
$$
(|\partial_t V_\eps|+|\nabla V_\eps|)+(|\partial_t \rho_\eps|+|\nabla \rho_\eps|)\le 8 \  {\rm{in}}\ \ P_1^+.
$$
This implies
$$\Big\|(\mathbb{I}_{l}+\rho_\eps \nabla_{V_\eps}\nu_N(V_\eps))-\mathbb{I}_l\Big\|_{L^\infty(P_1^+)} \le C\delta_N,$$
and 
$$
\Big\|\rho_\eps \langle \nabla^2_{V_\eps}\nu_N(V_\eps)(\nabla V_\eps,\nabla V_\eps), \nu_N(V_\eps)\rangle
-\rho_\eps \nabla_{V_\eps}\nu_N(V_\eps)(\nabla V_\eps, \nabla V_\eps)\Big\|_{L^\infty(P_1^+)}\le C.
$$
Hence by the $W^{2,1}_p$-estimate for linear parabolic equations, we obtain that 
$$\Big\|V_\eps\Big\|_{C^{1+\alpha}(P_\frac78^+)}\le C(\alpha), \ \forall \ \alpha\in (0,1),$$
uniformly with respect to $\eps$.

The boundary $C^{1+\alpha}$-estimate of $\rho_\eps$ can be done exactly as in Proposition 5.1. This completes the proof of
Lemma \ref{c1a}.
\end{proof}

Finally with Lemma \ref{c1a} at hand, we can show that $U_\eps$ also satisfies the gradient estimate as in Lemma 4.3. More precisely, we have
that

\begin{lemma}\label{eps_gradient_est1} There exists $\eps_0>0$,  depending only on $m$, such that if $U_\eps$ solves
\eqref{half-flow1} and satisfies, for $Z_0=(X_0, t_0)\in \partial\R^{m+1}_+
\times (0,\infty)$ and some $0<R<\frac{\sqrt{t_0}}2$, 
\begin{equation}\label{small_cond2}
\mathcal{E}(U_\eps, Z_0, R)<\eps_0^2,
\end{equation}
then
\begin{equation}\label{small_est3}
\displaystyle\sup_{P_{\delta_0 R}^+(Z_0)} R^2 |\nabla U_\eps|^2 \le C\delta_0^{-2},\ \ 
\displaystyle\sup_{P_{\delta_0 R}^+(Z_0)} R^4|\partial_t U_\eps|^2 \le C\delta_0^{-4},
\end{equation}
where $0<\delta_0<1$ and $C>0$ are independent of $\eps$.
\end{lemma}

\section*{Acknowledgements}
{ \small  AH was supported by the SNSF grants no. P400P2-183866 and P4P4P2-194460.    AS is member of the GNAMPA (Gruppo Nazionale per l'Analisi Matematica, la Probabilit\`a e le loro Applicazioni)
group of INdAM.  
AS acknowledges the partial support of the MIUR-PRIN Grant 2017 ``Variational methods for stationary and evolution problems with singularities and interfaces''.
CW is partially supported by NSF grants 1764417 and 2101224.
}

\bibliographystyle{acm}
\bibliography{biblio}

\begin{center}
Ali Hyder, TIFR Centre for Applicable Mathematics\\ Sharadanagar, Bangalore 560064, India\\ hyder@tifrbng.res.in
\end{center}

\bigskip
\begin{center}
Antonio Segatti,
Dipartimento di Matematica ``F. Casorati'', Universit\`a di Pavia\\
Via Ferrata 5, 27100 Pavia, Italy\\
antonio.segatti@unipv.it
\end{center}

\bigskip
\begin{center}
Yannick Sire, Department of Mathematics, Johns Hopkins University \\
3400 N. Charles Street, Baltimore, MD 21218, USA \\
ysire1@jhu.edu
\end{center}

\bigskip
\begin{center}
Changyou Wang, Department of Mathematics, Purdue University\\
150 N. University Street, West Lafayette, IN 47907, USA\\
wang2482@purdue.edu
\end{center}

\end{document}